\documentclass{article}
\usepackage[a4paper]{geometry}
\usepackage{amsfonts,amsmath,amsrefs,graphicx,amssymb,amsthm}

\newtheorem{theorem}{Theorem}[section]
\newtheorem{lemma}[theorem]{Lemma}
\newtheorem{corollary}[theorem]{Corollary}

\numberwithin{equation}{section}
\numberwithin{figure}{section}

\renewcommand{\geq}{\geqslant}
\renewcommand{\leq}{\leqslant}

\makeatletter
\renewcommand\section{\@startsection {section}{1}{\z@}%
                                   {-3.5ex \@plus -1ex \@minus -.2ex}%
                                   {1.3ex \@plus.2ex}%
                                   {\normalfont\large\scshape}}
\makeatother

\title{ \vspace{-5ex}\bf \Large Geodesic Rosen continued fractions}

\author{Ian Short and Mairi Walker\footnote{Department of Mathematics and Statistics, The Open University, Milton Keynes, MK7 6AA, United Kingdom 17~April~2015}}

\date{\vspace{-5ex}}

\begin{document}

\maketitle

\begin{abstract}
We describe how to represent Rosen continued fractions by paths in a class of graphs that arise naturally in hyperbolic geometry. This representation gives insight into Rosen's original work about words in Hecke groups, and it also helps us to identify  Rosen continued fraction expansions of shortest length.
\end{abstract}

\section{Introduction}\label{section 1}

In 1954, Rosen \cite{Ro1954} introduced a class of continued fractions now known as \emph{Rosen continued fractions} in order to study Hecke groups. Since then a rich literature on Rosen continued fractions has developed, including works on Diophantine approximation \cite{Le1985,Na2010,Ro1954}, the metrical theory of Rosen continued fractions \cite{BuKrSc2000,DaKrSt2009,KrScSm2010}, dynamics and geometry on surfaces associated to Hecke groups \cite{ArSc2013,MaMu2010,MaMuSt2012,MaSt2008,RoSc1992,ScSh1995}, and most recently on transcendence results for Rosen continued fractions \cite{BuHuSc2013}. This list of subjects and citations is by no means exhaustive. Here we describe how to represent Rosen continued fractions by paths in certain graphs of infinite valency that arise naturally in hyperbolic geometry. This perspective sheds light on Rosen's work, and allows us to tackle problems about the length of Rosen continued fractions, in a similar manner to the approach for integer continued fractions found in \cite{BeHoSh2012}.

Let us first introduce Hecke groups and discuss some of their geometric properties, before defining Rosen continued fractions. For a more detailed exposition of Hecke groups, the reader should consult \cite{BeKn2008,Ev1973,Le1964,ScSh1997}. Given an integer $q\geq 3$, the \emph{Hecke group}\, $G_q$ is the Fuchsian group generated by the M\"obius transformations
\[
\sigma(z)=-\frac{1}{z}\quad\text{and}\quad \tau_q(z)=z+\lambda_q,
\]
where $\lambda_q=2\cos (\pi/q)$. (From now on, for simplicity, we omit the subscript $q$ from $\tau_q$ and other elements of $G_q$ that depend on $q$.) The generators $\sigma$ and $\tau$ satisfy the relations $\sigma^2=(\tau\sigma)^q=I$, where $I$ is the identity transformation, and all other relations in $\sigma$ and $\tau$ are consequences of these two. It follows that $G_q$ is isomorphic as a group to the free product of cyclic groups $C_2*C_q$. Of particular importance among the Hecke groups is $G_3$, which is the \emph{modular group}, that is, the group of M\"obius transformations $z\mapsto (az+b)/(cz+d)$, where $a$, $b$, $c$, and $d$ are integers and $ad-bc=1$.

It is often more convenient to work with an alternative pair of generators of $G_q$, namely $\tau$ and $\rho$, where
\[
\rho(z)=\tau\sigma(z)=\lambda_q-\frac{1}{z}.
\]
The map $\rho$ is an elliptic M\"obius transformation of order $q$ with fixed points $e^{i\pi/q}$ and $e^{-i\pi/q}$. In the upper half-plane $\mathbb{H}$, which is a standard model of the hyperbolic plane, the hyperbolic quadrilateral $D$ that has vertices $i$, $e^{i\pi/q}$, $\lambda_q+i$, and $\infty$ (an ideal vertex) is a fundamental domain for $G_q$, with side-pairing transformations $\tau$ and $\rho$. The quadrilateral $D$, with $q=5$, is shown in Figure~\ref{figure 1}(a).

\begin{figure}[ht]
\centering
\includegraphics{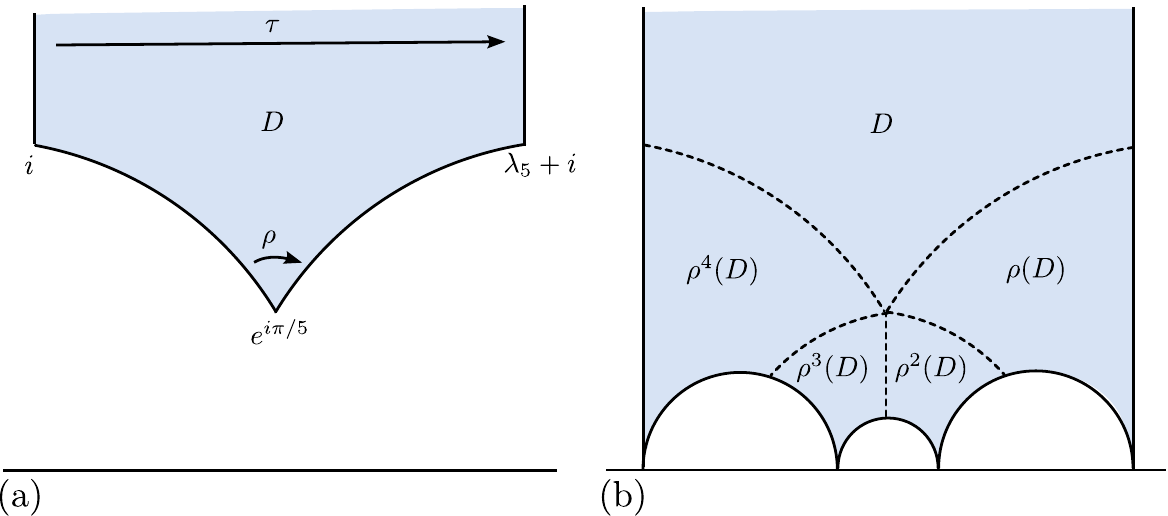}
\caption{Fundamental domains for $G_5$ (left) and $\Gamma_5$ (right)}
\label{figure 1}
\end{figure}

Let $\Gamma_q$ denote the group generated by the involutions $\rho^i \sigma \rho^{-i}$, $i=0,\dots,q-1$, which is a normal subgroup of $G_q$ of index $q$. A fundamental domain $E$ for $\Gamma_q$ is given by $E=\bigcup_{i=0}^{q-1}\rho^i(D)$, as shown in Figure~\ref{figure 1}(b). This fundamental domain is an ideal hyperbolic $q$-gon and its images under $\Gamma_q$ tessellate the hyperbolic plane by ideal hyperbolic $q$-gons. The skeleton of this tessellation is a connected plane graph, which we call a \emph{Farey graph}, and denote by $\mathcal{F}_q$. The vertices of $\mathcal{F}_q$ are the ideal vertices of the tessellation; they all belong to the ideal boundary $\mathbb{R}\cup\{\infty\}$ of $\mathbb{H}$, and in fact they form a countable, dense  subset of $\mathbb{R}\cup\{\infty\}$. They are the full collection of parabolic fixed points of $G_q$. The edges of $\mathcal{F}_q$ are the sides of the ideal $q$-gons in the tessellation, and the faces of $\mathcal{F}_q$ are the ideal  $q$-gons themselves. Part of $\mathcal{F}_5$ is shown in Figure~\ref{figure 2}.

\begin{figure}[ht]
\centering
\includegraphics{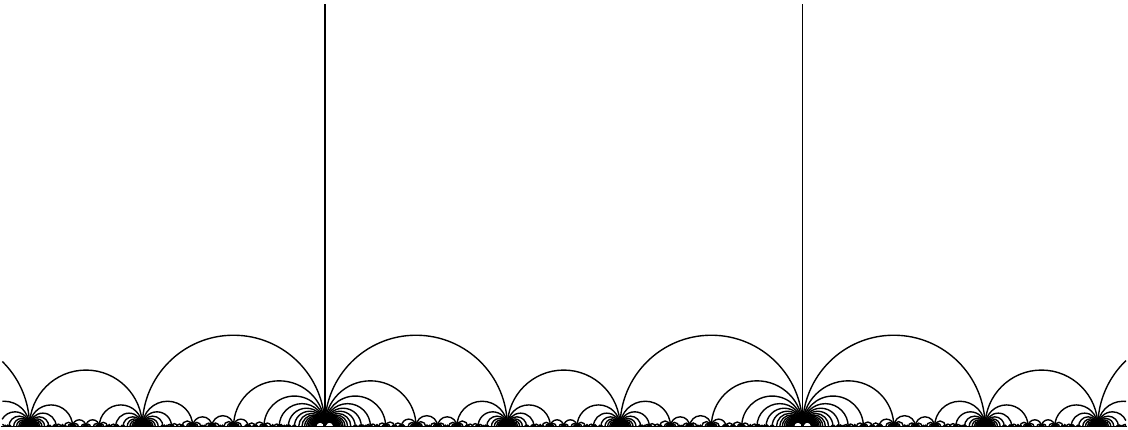}
\caption{The Farey graph $\mathcal{F}_5$}
\label{figure 2}
\end{figure}

The term ``Farey graph'' is motivated by the $q=3$ case because the graph $\mathcal{F}_3$ is often given that name. The graph $\mathcal{F}_3$ is the skeleton of a tessellation of the hyperbolic plane by ideal triangles, the vertices of which are the rational numbers and $\infty$. It has been used already to study continued fractions, in works such as \cite{BeHoSh2012,KaUg2007,Mo1982,Se1985} and \cite[Chapter~19]{Sc2011}.  The Farey graphs (for all values of $q$) also arise in subjects involving hyperbolic geometry that are not directly related to continued fractions; for example, they form a class of universal objects in the theory of maps on surfaces (see \cite{IvSi2005,JoSiWi1991,Si1988}).

Rosen \cite{Ro1954} observed that words in the generators $\sigma$ and $\tau$ of $G_q$ give rise to continued fractions of a particular type, which can be used to study $G_q$. These continued fractions, which are now known as  \emph{Rosen continued fractions}, are expressions of the form
\[
b_1\lambda_q +\cfrac{-1}{b_2\lambda_q
          + \cfrac{-1}{\raisebox{-1ex}{$b_3\lambda_q+\dotsb$}}}\,,
\]
where the entries $b_i$ are integers. This Rosen continued fraction is said to be either \emph{infinite} or \emph{finite} depending on whether the sequence $b_1,b_2,\dotsc$ is infinite or finite, respectively. We denote it by $[b_1,b_2,\dotsc]_q$ in the former case, and by $[b_1,\dots,b_n]_q$ in the latter case. The \emph{value} of a finite Rosen continued fraction $[b_1,\dots,b_n]_q$ is the number that you obtain by evaluating that expression. The \emph{convergents} of a finite or infinite Rosen continued fraction are the values of $[b_1,\dots,b_m]_q$ for $m=1,2,\dotsc$. If the sequence of convergents of an infinite Rosen continued fraction converges to a point $x$ in $\mathbb{R}\cup\{\infty\}$, then we say that the continued fraction \emph{converges} and has \emph{value} $x$. Sometimes we abuse notation and write $[b_1,b_2,\dotsc]_q$ and $[b_1,\dots,b_n]_q$ for the values of the continued fractions (rather than the formal expressions that they are supposed to represent).  There are various other definitions of Rosen continued fractions in the literature, all similar to our definition, and some equivalent. When $q=3$ we see that $\lambda_q=1$, so the corresponding Rosen continued fractions have integer coefficients `along the bottom''. 

The main purpose of this paper is to describe a correspondence between paths in Farey graphs and Rosen continued fractions. In more detail, consider a path in $\mathcal{F}_q$ that starts at $\infty$ and ends at some vertex $y$, such as the path shown in Figure~\ref{figure 3}. %
\begin{figure}[ht]
\centering
\includegraphics{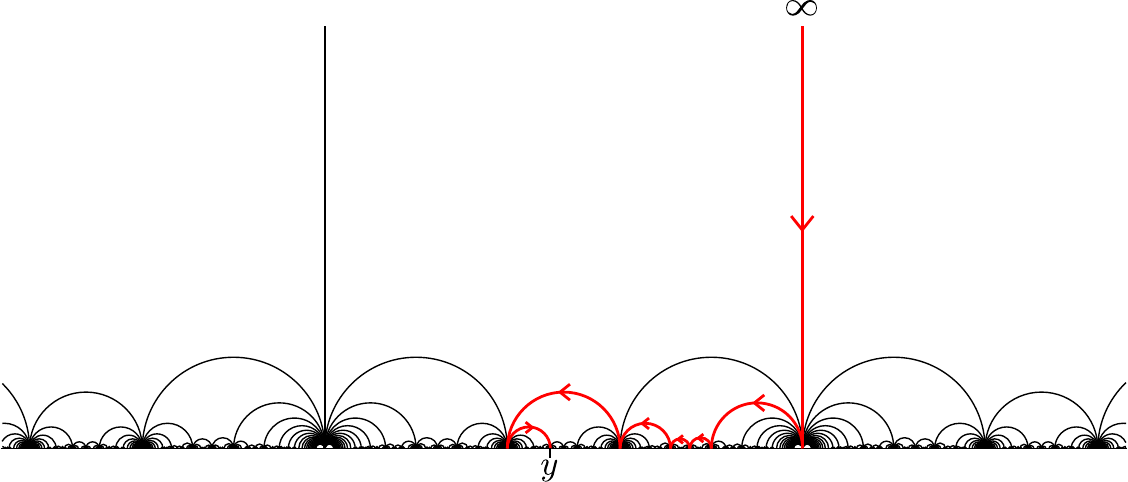}
\caption{A path in $\mathcal{F}_5$ from $\infty$ to a vertex $y$}
\label{figure 3}
\end{figure}%
Let the vertices of the path be $\infty,v_1,v_2,\dots,v_n=y$, in that order. We will show (Theorem~\ref{theorem 4}) that there is a unique finite Rosen continued fraction whose convergents are the vertices $v_1,\dots,v_n$. Some of Rosen's techniques from \cite{Ro1954} can be explained geometrically using this correspondence, as we will see later.  We also establish some results about ``shortest'' Rosen continued fractions, with an approach similar to that found in \cite{BeHoSh2012} for integer continued fractions. We now discuss these results.

 A Rosen continued fraction \emph{expansion} of a number $x$ is a Rosen continued fraction with value $x$. Each real number has infinitely many Rosen continued fraction expansions. For example, in the familiar case $q=3$ in which  Rosen continued fractions have integer coefficients, there are numerous algorithms that give rise to different continued fraction expansions. The most well known of these is Euclid's algorithm, which gives the regular continued fraction expansion with positive integer coefficients. A similar algorithm is the nearest-integer algorithm (which we will discuss in detail in Section~\ref{section 3}). It has long been known (see \cite[page 168]{Pe1950}) that among all integer continued fraction expansions of a rational number, the one arising from the nearest-integer algorithm has the least number of terms.

Rosen observed that there is a version of the nearest-integer algorithm that gives rise to Rosen continued fraction expansions. Our first result says (in part) that applying this algorithm to a vertex $y$ of $\mathcal{F}_q$ gives a Rosen continued fraction expansion of $y$ with the least possible number of terms. We say that a finite Rosen continued fraction  $[b_1,\dots,b_n]_q$ with value  $y$ is a \emph{geodesic} Rosen continued fraction if every other Rosen continued fraction expansion of $y$ has at least $n$ terms. We use the phrase ``geodesic'' because $[b_1,\dots,b_n]_q$ corresponds to a path in $\mathcal{F}_q$ between $\infty$ and $y$ with the least possible number of edges (a geodesic path). We say that an infinite Rosen continued fraction $[b_1,b_2,\dotsc]_q$ is a \emph{geodesic} Rosen continued fraction if $[b_1,\dots,b_m]_q$ is a geodesic Rosen continued fraction for $m=1,2,\dotsc$.

\begin{theorem}\label{theorem 1}
For each integer $q\geq 3$, the nearest-integer algorithm applied to any real number gives rise to a geodesic Rosen continued fraction.
\end{theorem}

Our next theorem gives bounds on the maximum number of finite geodesic Rosen continued fraction expansions of a vertex of $\mathcal{F}_q$. To explain the notation of this theorem, we first introduce informally a concept that will later be made precise. Given any vertex $y$ of $\mathcal{F}_q$, we can ``shade in'' each of the faces of $\mathcal{F}_q$ that separates $\infty$ from $y$, as shown in Figure~\ref{figure 7}. %
\begin{figure}[ht]
\centering
\includegraphics{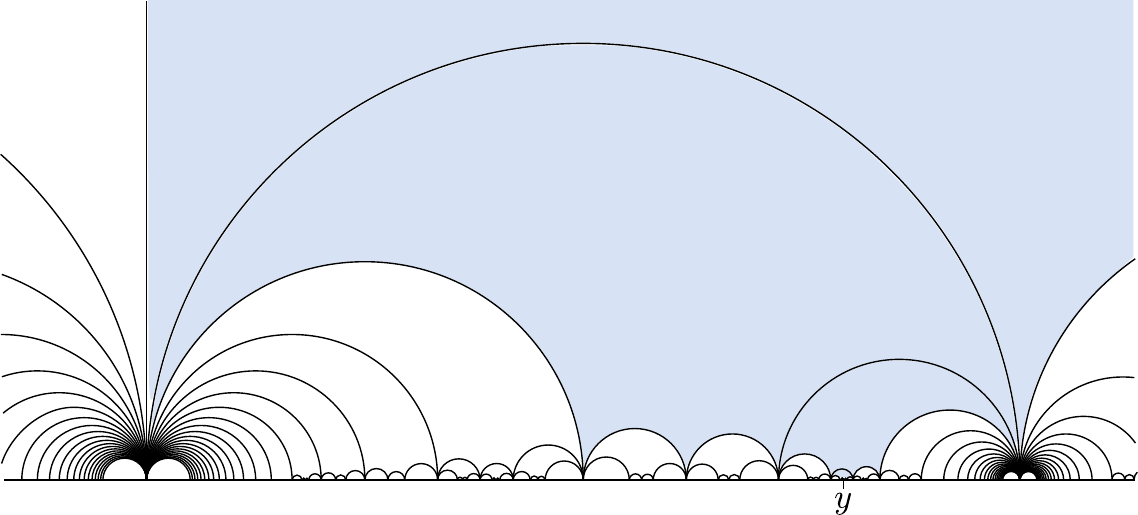}
\caption{The ideal $q$-gons that separate $\infty$ and $y$ are shaded}
\label{figure 7}
\end{figure}%
This results in a chain of $q$-gons, with $\infty$ a vertex of the first $q$-gon and $y$ a vertex of the last $q$-gon. There are other ways of describing the sequence of $q$-gons that arises in the chain; for instance, they are exactly those $q$-gons in the tessellation that intersect the vertical hyperbolic line between $\infty$ and~$y$.

It is often clearer to represent the chain of $q$-gons by polygons of a similar Euclidean size, such as those shown in Figure~\ref{figure 6} (the first four $q$-gons of this chain match those of Figure~\ref{figure 7}).

\begin{figure}[ht]
\centering
\includegraphics{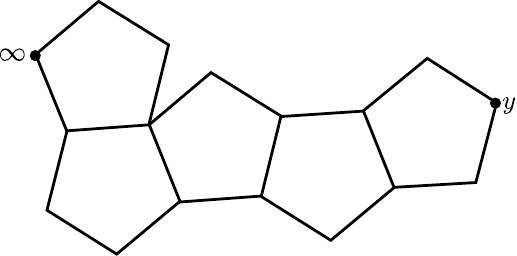}
\caption{A chain of $5$-gons connecting $\infty$ and $y$}
\label{figure 6}
\end{figure}

We define $D(\infty,y)$ to be the number of $q$-gons in this chain. We also define $F_n$ to be the $n$th term of the Fibonacci sequence, which is given by $F_0=1$, $F_1=2$, $F_2$=3, $F_3=5$, and so forth (note the unusual choice of indices).

\begin{theorem}\label{theorem 2}
Suppose that $y$ is a vertex of $\mathcal{F}_q$, and $D(\infty,y)=n$. Then there are at most $F_{n}$ finite geodesic Rosen continued fraction expansions of $y$.
\end{theorem}

We prove this theorem in Section~\ref{section 6}, and show that the bound $F_{n}$ can be attained when $n$ is even. For odd values of $n$, we will give better bounds than~$F_{n}$.

Our third theorem  gives necessary and sufficient conditions for $[b_1,\dots,b_n]_q$ to be a geodesic Rosen continued fraction. For now we only state a result in which $q$ is even; a similar if slightly more complicated theorem when $q$ is odd is given in Section~\ref{section 7a}. To  formulate our result concisely, we use the notation $1^{[d]}$ to mean the sequence consisting of $d$ consecutive $1$s. Also, given a sequence $x_1,\dots,x_n$, we write $\pm (x_1,\dots,x_n)$ to mean one of the two sequences $x_1,\dots,x_n$ or $-x_1,\dots,-x_n$.

\begin{theorem}\label{theorem 3}
Suppose that $q=2r$, where $r\geq 2$. The continued fraction $[b_1,\dots,b_n]_q$ is a geodesic Rosen continued fraction if and only if  the sequence $b_2,\dots,b_n$ has no terms equal to $0$ and contains no subsequence of consecutive terms either of the form $\pm  1^{[r]}$ or of the form
\[
\pm (1^{[r-1]},2,1^{[r-2]},2,1^{[r-2]},\dots,1^{[r-2]},2,1^{[r-1]}).
\]
\end{theorem}

The number of $2$s in the sequence above can be any positive integer.

In geometric terms, the theorem says that the path corresponding to a Rosen continued fraction is a geodesic path unless it either doubles back on itself or takes ``the long way round'' the outside of a chain of $q$-gon (see Sections~\ref{section 7a} and~\ref{section 7b}). These possibilities are illustrated when $q=4$ in Figure~\ref{figure 19}.%
\begin{figure}[ht]
\centering
\includegraphics{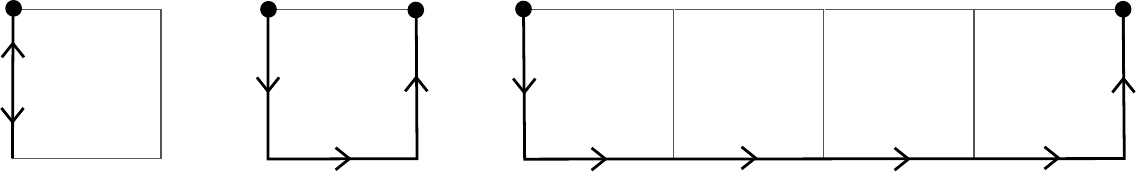}
\caption{A path in $\mathcal{F}_4$ that is not a geodesic path must contain a subpath of type similar to one of these.}
\label{figure 19}
\end{figure}

Sections~\ref{section 2} to \ref{section 7b} are about finite Rosen continued fractions, but in Section~\ref{section 8} we turn our attention to infinite Rosen continued fractions. By representing continued fractions by paths in Farey graphs we obtain the following theorem, which gives remarkably mild sufficient conditions for an infinite Rosen continued fraction to converge.

\begin{theorem}\label{theorem 13}
If the sequence of convergents of an infinite Rosen continued fraction does not contain infinitely many terms that are equal, then the continued fraction converges.
\end{theorem}

There is an obvious converse to this theorem: if an infinite Rosen continued fraction converges to some value $x$, then it can only have infinitely many convergents equal to some vertex $y$ if $x=y$. For example, $0,1,0,1/2,0,1/3,\dotsc$ is the sequence of convergents of the integer continued fraction $[0,-1,0,-2,0,-3,\dotsc]_3$, and this sequence converges to the vertex $0$ of $\mathcal{F}_3$.

\section{Paths in Farey graphs}\label{section 2}

In this section we describe in more detail the correspondence between Rosen continued fractions and paths in Farey graphs. The procedure is similar to that found in \cite{BeHoSh2012}. For now we concentrate only on finite continued fractions and finite paths; near the end of the paper we will consider infinite continued fractions and paths.

Let us begin by  introducing some notation and terminology from  graph theory. In any graph $X$, two vertices $u$ and $v$ that are connected by an edge are said to be \emph{adjacent} or \emph{neighbours}, and we write $u\sim v$. We denote the edge incident to $u$ and $v$ by $\{u,v\}$ (or $\{v,u\}$). A \emph{path} in $X$ is a sequence of vertices $v_0,v_1,\dots,v_n$, where  $v_{i-1}\sim v_i$ for $i=1,\dots,n$. We represent this path by $\langle v_0,v_1,\dots,v_n\rangle$. The pairs $\{v_{i-1},v_i\}$ are called the \emph{edges} of the path. We define the \emph{length} of the path to be $n$ (this is one less than the number of vertices; it corresponds to the number of edges of the path). A \emph{subpath} of $\langle v_0,v_1,\dots,v_n\rangle$ is a path of the form $\langle v_i,\dots,v_j\rangle$, where $0\leq i<j\leq n$.

We focus on the Farey graphs $\mathcal{F}_q$, which have infinitely many vertices, and each vertex has infinitely many neighbours. There is an alternative way to define $\mathcal{F}_q$ to that given in the introduction. Let $L$ be the hyperbolic line in the upper half-plane $\mathbb{H}$ between (and including) $0$ and $\infty$. Under iterates of the map $\rho(z)=\lambda_q-1/z$, this hyperbolic line is mapped to each of the $q$ sides of the fundamental domain $E$ of the normal subgroup $\Gamma_q$ of $G_q$ which is shown in Figure~\ref{figure 1}(b). Since $G_q=\bigcup_{i=0}^{q-1} \Gamma_q\rho^i$, it follows that  $\mathcal{F}_q$ is the orbit of $L$ under $G_q$, and it could have been defined in this way. The transformation $\sigma(z)=-1/z$ maps $\infty$ to $0$, so the set of vertices of $\mathcal{F}_q$ is the orbit of $\infty$ under $G_q$. 

Using this description of $\mathcal{F}_q$, we can determine the neighbours of $\infty$ in $\mathcal{F}_q$. Let $\Lambda_q$ be the stabiliser of $\infty$ in $G_q$; this is the cyclic group generated by $\tau(z)=z+\lambda_q$. A vertex $v$ is a neighbour of $\infty$ if and only if $v=g(0)$ for some element $g$ of $\Lambda_q$. Therefore the neighbours of $\infty$ are the integer multiples of $\lambda_q$.

We can also use this alternative description of Farey graphs to determine the automorphism groups of these graphs. Here we consider an \emph{automorphism} of a graph $X$ to be a bijective map $f$ of the vertices of $X$ such that two vertices $u$ and $v$ are adjacent if and only if $f(u)$ and $f(v)$ are adjacent. Since $\mathcal{F}_q$ is the orbit of $L$ under $G_q$, it follows at once that each element of $G_q$ induces an automorphism of $\mathcal{F}_q$. The map $\kappa(z)=-\bar{z}$ also induces an automorphism of $\mathcal{F}_q$. It is an anticonformal transformation, so it reverses the cyclic order or vertices around faces, whereas the conformal transformations in $G_q$ preserve the cyclic order of vertices around faces. The group generated by $\kappa$ and $G_q$ is in fact the full group of automorphisms of $\mathcal{F}_q$ (we omit proof of this observation as we do not need it). 

We are now in a position to state the theorem that explains the correspondence between finite Rosen continued fractions and finite paths in Farey graphs. The $q=3$ case has been established already, in \cite[Theorem~3.1]{BeHoSh2012}, and our proof is similar, so we only sketch the details. In our sketch proof, and elsewhere in the paper, we use the following notation. Given a Rosen continued fraction $[b_1,\dots,b_n]_q$, we define a sequence of M\"obius transformations
\[
s_i(z) = b_i\lambda_q-\frac{1}{z},\quad\text{for $i=1,\dots,n$}.
\]
That is, $s_i=\tau^{b_i}\sigma$, using the generators $\sigma$ and $\tau$ of $G_q$. The convergents of $[b_1,\dots,b_n]_q$ are then given by $s_1\dotsb s_m(\infty)$, for $m=1,\dots,n$. (Here $s_1\dotsb s_m$ represents the composition of the functions $s_1,\dots,s_m$.)

\begin{theorem}\label{theorem 4}
Let $y$ be a vertex of $\mathcal{F}_q$.  The vertices $v_1,\dots,v_n$ of $\mathcal{F}_q$, with $v_n=y$, are the consecutive convergents of some Rosen continued fraction expansion of $y$ if and only if $\langle \infty,v_1,\dots,v_n\rangle$ is a path in $\mathcal{F}_q$ from $\infty$ to $y$.
\end{theorem}
\begin{proof}[Sketch proof]
Suppose first that $v_1,\dots,v_n$ are the consecutive convergents of the Rosen continued fraction expansion $[b_1,\dots,b_n]_q$ of $y$. Then, for $m=2,\dots,n$,  $v_{m-1}$ and $v_m$ are the images of $0$ and $\infty$, respectively, under the transformation $s_1\dotsb s_m$ from $\mathcal{F}_q$, so they are adjacent. Therefore  $\langle \infty,v_1,\dots,v_n\rangle$ is a path in $\mathcal{F}_q$.

Conversely, suppose that $\langle\infty,v_1,\dots,v_n\rangle$ is a path from $\infty$ to $y$. We can construct a Rosen continued fraction $[b_1,\dots,b_n]_q$ inductively by defining $b_1=v_1/\lambda_q$ and $b_m=s_{m-1}^{-1}\dotsb s_1^{-1}(v_m)/\lambda_q$ for $m=2,\dots,n$. The convergents of this continued fraction are $v_1,\dots,v_n$.
\end{proof}

We call $\langle \infty, v_1,\dots,v_n\rangle$ the \emph{path of convergents} of $[b_1,\dots,b_n]_q$. There is a simple way to move between a Rosen continued fraction and its path of convergents, which we outline here, and which is illustrated in Figure~\ref{figure 8}. The integers $b_2,\dots,b_n$ of the expansion $[b_1,\dots,b_n]_q$ of a vertex $y$ encode a set of directions to navigate the corresponding path $\langle \infty, v_1,\dots,v_n\rangle$. In brief, to navigate the path you should upon reaching $v_{i-1}$ take the ``$b_i$th right turn'' to get to $v_i$ (which is a left turn if $b_i$ is negative).

\begin{figure}[ht]
\centering
\includegraphics{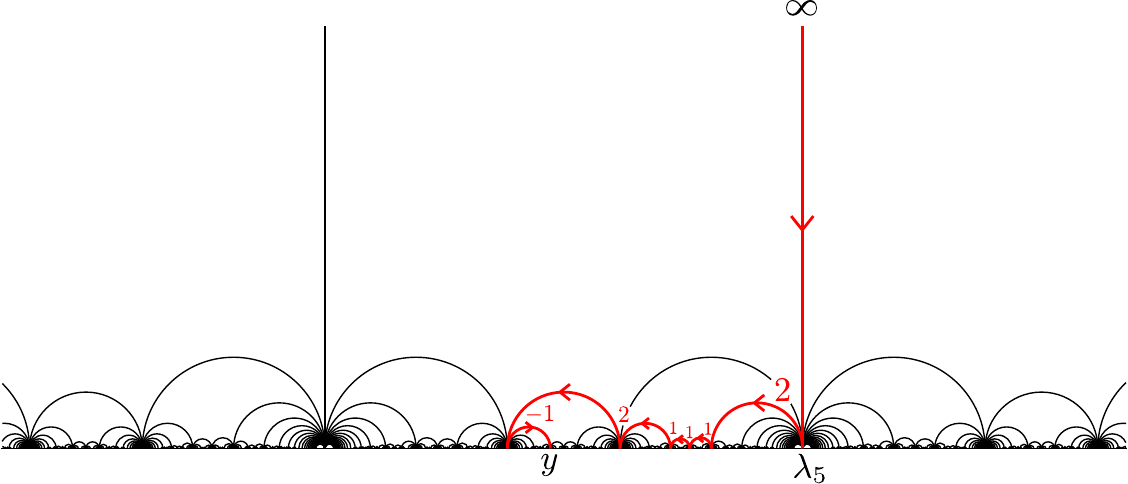}
\caption{The route of the path tells us that $y=[1,2,1,1,1,2,-1]_5$}
\label{figure 8}
\end{figure}

Let us describe this procedure in more detail. Suppose that $a$, $b$, and $c$ are vertices of $\mathcal{F}_q$ such that $a\sim b$ and $b\sim c$. We are going to define an integer-valued function $\phi(a,b,c)$. Suppose first that $b=\infty$; then $\phi(a,\infty,c)=(c-a)/\lambda_q$. Now suppose that $b\neq\infty$. In this case we choose an element $f$ of $G_q$ such that $f(b)=\infty$ and define $\phi(a,b,c)=\phi(f(a),f(b),f(c))$. The choice of $f$ does not matter, because if $g$ is another element of $G_q$ such that $g(b)=\infty$, then $g=\tau^mf$ for some integer $m$, so
\begin{align*}
\phi(g(a),g(b),g(c)) &=(g(c)-g(a))/\lambda_q\\
&=((m\lambda_q+f(c))-(m\lambda_q+f(a)))/\lambda_q\\
&=(f(c)-f(a))/\lambda_q\\
&=\phi(f(a),f(b),f(c)).
\end{align*}
A consequence of this definition is that $\phi$ is invariant under elements of $G_q$, in the sense that
\[
\phi(f(a),f(b),f(c))=\phi(a,b,c)
\]
for any map $f$ in $G_q$ and three vertices $a$, $b$, and $c$.

The function $\phi$ has a simple geometric interpretation when $b\neq\infty$, which we obtain by mapping $b$ (conformally) to $\infty$ by a member of $G_q$. Label the edges incident to $b$ by the integers, with $\{a,b\}$ (the edge between $a$ and $b$) labelled $0$, and the other edges labelled in anticlockwise order around $b$. Then $\phi(a,b,c)$ is the integer label for the edge $\{b,c\}$. An example is shown in Figure~\ref{figure 10}

\begin{figure}[ht]
\centering
\includegraphics{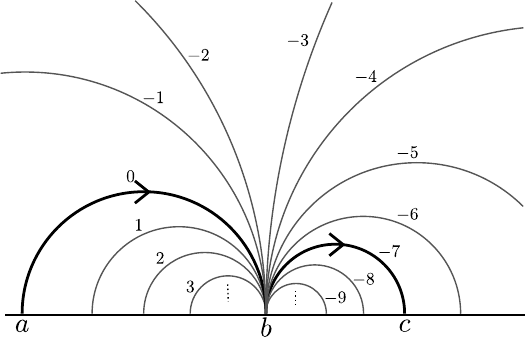}
\caption{Example in which $\phi(a,b,c)=-7$}
\label{figure 10}
\end{figure}

The first vertex of the path of convergents of a Rosen continued fraction $[b_1,\dots,b_n]_q$ is $\infty$, and the second vertex is $b_1\lambda_q$. The next lemma explains precisely how the integers $b_2,\dots,b_n$ encode a set of directions for navigating the remainder of the path.

\begin{lemma}\label{lemma 7}
Let $\langle v_0,v_1,\dots,v_n\rangle$, where $v_0=\infty$ and $n\geq 2$, be the path of convergents of the Rosen continued fraction $[b_1,\dots,b_n]_q$. Then 
\[
\phi(v_{i-2},v_{i-1},v_{i})=b_{i}
\]
for $i=2,\dots,n$.
\end{lemma}
\begin{proof}
Since $\phi$ is invariant under the element $s_1\dotsb s_{i-1}$ of $G_q$, we see that
\begin{align*}
\phi(v_{i-2},v_{i-1},v_{i}) &= \phi(s_1\dotsb s_{i-2}(\infty),s_1\dotsb s_{i-1}(\infty),s_1\dotsb s_{i}(\infty))\\
&= \phi(s_{i-1}^{-1}(\infty),\infty,s_i(\infty))\\
&= \phi(0,\infty,b_i\lambda_q)\\
&= b_i. \qedhere
\end{align*}
\end{proof}

\section{The nearest-integer algorithm}\label{section 3}

One way to obtain a Rosen continued fraction expansion of a real number  is to apply an algorithm known as the \emph{nearest-integer algorithm}, which was referred to in the introduction. It is much the same as the more familiar nearest-integer algorithm that is used with integer continued fractions. The version for Rosen continued fractions was supplied  by Rosen in \cite{Ro1954}, and similar algorithms can be found in \cite{MaMu2010,MaSt2008,Na1995}.  Here we describe an algorithm for constructing a path of shortest possible length between two vertices of a Farey graph (a geodesic path). Towards the end of the section we then show that when one of the vertices is $\infty$, the algorithm reduces to the nearest-integer algorithm (thereby proving Theorem~\ref{theorem 1}, at least for finite Rosen continued fractions). When $q=3$, our algorithm coincides with that of \cite[Section~9]{BeHoSh2012}.

We begin by introducing some more standard concepts and terminology from graph theory. Given vertices $x$ and $y$ of a connected graph $X$, a \emph{geodesic path} from $x$ to $y$ is a path between these two vertices of shortest possible length. We define a metric $d$ on $X$ called the \emph{graph metric}, where $d(x,y)$ is the length of any geodesic path from $x$ to
$y$. We denote the graph metric on $\mathcal{F}_q$ by $d_q$. 

As we have seen, the edges of the Farey graph $\mathcal{F}_q$ lie in the upper half-plane $\mathbb{H}$ and the vertices of $\mathcal{F}_q$ lie on the ideal boundary of $\mathbb{H}$, namely $\mathbb{R}\cup\{\infty\}$, which from now on we denote more simply by $\mathbb{R}_\infty$. We can map $\mathbb{H}$ conformally on to the unit disc by a M\"obius transformation, and under such a transformation the ideal boundary $\mathbb{R}_\infty$ maps to the unit circle. It is often more convenient to think of $\mathcal{F}_q$ as a graph in the unit disc, not least because in that model it is obvious geometrically that the ideal boundary is topologically a circle, and we can speak of a finite list of points on the circle occurring in ``clockwise order''.

The M\"obius transformation we use to transfer $\mathcal{F}_q$ to the unit disc is 
\[
\psi(z)=\frac{z-e^{i\pi/q}}{z-e^{-i\pi/q}},
\]
 because this maps $e^{i\pi/q}$, one of the centres of rotation of the generator $\rho$ of $G_q$, to $0$. Part of the Farey graph $\mathcal{F}_5$ is shown in the unit disc in Figure~\ref{figure 11}.

\begin{figure}[ht]
\centering
\includegraphics{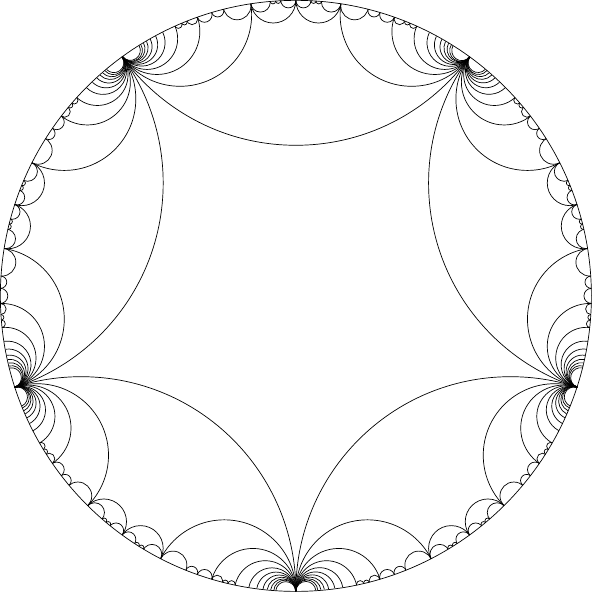}
\caption{The Farey graph $\mathcal{F}_5$}
\label{figure 11}
\end{figure}

At the heart of our algorithm for constructing a geodesic path in $\mathcal{F}_q$ is the following  lemma.

\begin{lemma}\label{lemma 2}
Let $x$ and $y$ be two non-adjacent vertices of $\mathcal{F}_q$. Among the faces of $\mathcal{F}_q$ that are incident to $x$, there is a unique one $P$ such that if $u$ and $v$ are the two vertices of $P$ that are adjacent to $x$, then $y$ belongs to the component of $\mathbb{R}_\infty\setminus\{u,v\}$ that does not contain $x$.
\end{lemma}
\begin{proof}
We prove the lemma in the upper half-plane model where, after applying an element of $G_q$, we may assume that $x=\infty$ and $y\in(0,\lambda_q)$. With this choice of $x$ and $y$, the unique polygon $P$ is the fundamental domain $E$ of the normal subgroup $\Gamma_q$ of $G_q$ defined in the introduction, which is shown in Figure~\ref{figure 1}(b).
\end{proof}

We denote the polygon $P$ described in Lemma~\ref{lemma 2} by $P_y(x)$; an example is shown in Figure~\ref{figure 4}.

\begin{figure}[ht]
\centering
\includegraphics{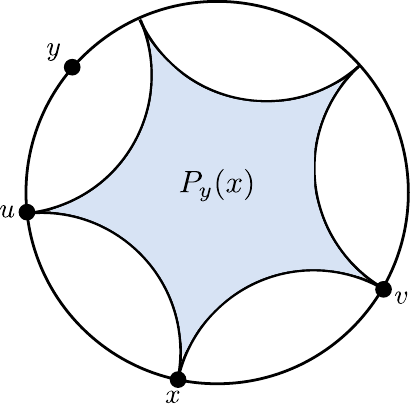}
\caption{A polygon $P_y(x)$ in $\mathcal{F}_5$}
\label{figure 4}
\end{figure}

We call the two vertices $u$ and $v$ described in Lemma~\ref{lemma 2} the \emph{$y$-parents} of $x$. When $x$ and $y$ are adjacent or equal, we define the $y$-parents of $x$ to both equal $y$ (so really there is only one $y$-parent in each of these cases). The importance of the concept of $y$-parents can be seen in the following sequence of four results.

\begin{lemma}\label{lemma 6}
Let $u$ and $v$ be distinct vertices of a face $P$ of $\mathcal{F}_q$. Suppose that $\gamma$ is a path in $\mathcal{F}_q$ that starts at a vertex in one component of $\mathbb{R}_\infty\setminus\{u,v\}$ and finishes at a vertex in the other component of $\mathbb{R}_\infty\setminus\{u,v\}$. Then $\gamma$ passes through one of $u$ or $v$.
\end{lemma}
\begin{proof}
If the lemma is false, then there is an edge  of $\gamma$ with end points in each of the components of $\mathbb{R}_\infty\setminus\{u,v\}$. This edge intersects the hyperbolic line between $u$ and $v$, which is a contradiction, because this hyperbolic line either lies inside the face $P$ or else it is an edge of $P$.
\end{proof}

\begin{theorem}\label{theorem 6}
Any path between vertices $x$ and $y$ of $\mathcal{F}_q$ must pass through one of the $y$-parents of $x$.
\end{theorem}
\begin{proof}
This is certainly true if $x$ and $y$ are equal or adjacent, as in this case the $y$-parents of $x$ both equal $y$. Otherwise, the theorem follows immediately from Lemmas~\ref{lemma 2} and \ref{lemma 6}.
\end{proof}

\begin{corollary}\label{corollary 1}
If $\langle v_0,\dots,v_n\rangle$ is a geodesic path in $\mathcal{F}_q$ with $v_0=x$ and $v_n=y$, where $n\geq 1$, then $v_1$ is a $y$-parent of $x$.
\end{corollary}
\begin{proof}
By Theorem~\ref{theorem 6}, we can choose an integer $i\geq 1$ such that $v_i$ is a $y$-parent of $x$, and we can assume that $i$ is the smallest such integer. Then $i$ must be $1$, otherwise the path $\langle v_0,\dots,v_n\rangle$ is not a geodesic path because $\langle v_0,v_i,v_{i+1},\dots,v_n\rangle$ is a shorter path between $x$ and $y$.
\end{proof}

Corollary~\ref{corollary 1} itself has a corollary that we will need soon.

\begin{corollary}\label{corollary 3}
Each geodesic path between vertices $x$ and $y$ of a face $P$ of $\mathcal{F}_q$ is given by traversing the edges of $P$. There is a unique geodesic path between $x$ and $y$ unless $q$ is even and $x$ and $y$ are opposite vertices of $P$, in which case there are exactly two geodesic paths between $x$ and $y$.
\end{corollary}
\begin{proof}
This is clearly true if $x$ and $y$ are adjacent. Otherwise, the $y$-parents of $x$ both lie in $P$, by Lemma~\ref{lemma 2}. It then follows from Corollary~\ref{corollary 1} that any geodesic path between $x$ and $y$ is confined to the vertices of $P$, and the result follows immediately.
\end{proof}

We denote the $y$-parents of $x$ by $\alpha_y(x)$ and $\beta_y(x)$ in some order that we now explain. If $x$ and $y$ are equal or adjacent, then we must define $\alpha_y(x)$ and $\beta_y(x)$ to both equal $y$. For the remaining possibilities, we split our discussion in to two cases, depending on whether $q$ is even or odd.

Suppose first that $q$ is even. In this case there is a vertex $w$ of the $q$-gon $P_y(x)$ defined by Lemma~\ref{lemma 2} that is opposite $x$. If $y=w$, then we define $\alpha_y(x)$ and $\beta_y(x)$ to be such that the vertices $\alpha_y(x),x,\beta_y(x)$ lie in that order clockwise around $\mathbb{R}_\infty$.  If $y\neq w$, then we define $\alpha_y(x)$ to be whichever of the $y$-parents $u$ and $v$ of  $x$ lies in the same component of $\mathbb{R}_\infty\setminus\{x,w\}$ as $y$, as shown in Figure~\ref{figure 5}(a). Of course, $\beta_y(x)$ is then the remaining vertex $u$ or $v$.%
\begin{figure}[ht]
\centering
\includegraphics{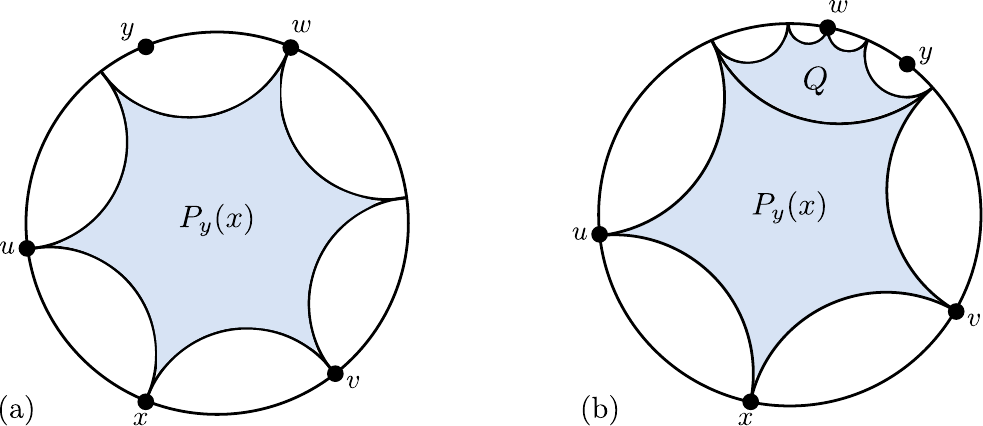}
\caption{(a) The vertex $u$ is $\alpha_y(x)$ (b) The vertex $v$ is $\alpha_y(x)$}
\label{figure 5}
\end{figure}

Suppose now that $q$ is odd. In this case, there is an edge of the $q$-gon $P_y(x)$ that is opposite $x$, rather than a vertex. Let $Q$ be the other face of $\mathcal{F}_q$ that is incident to that edge. Together $P_y(x)$ and $Q$ form a $2q$-gon, and we now define $\alpha_y(x)$ and $\beta_y(x)$ using this $2q$-gon in the same way that we did when $q$ was even, as shown in Figure~\ref{figure 5}(b).


We have now defined two maps $\alpha_y:\mathcal{F}_q\to\mathcal{F}_q$ and $\beta_y:\mathcal{F}_q\to \mathcal{F}_q$. By definition, they are invariant under $G_q$ in the sense that $f(\alpha_y(x))=\alpha_{f(y)}(f(x))$ and $f(\beta_y(x))=\beta_{f(y)}(f(x))$ for any transformation $f$ from $G_q$. When $q=3$ and $y=\infty$, the vertices $\alpha_y(x)$ and $\beta_y(x)$ have been given various names in continued fractions literature. In \cite{BeHoSh2012} they were called the \emph{first parent} and \emph{second parent} of $x$ and in \cite{LaTr1995} they were called the \emph{old parent} and \emph{young parent} of $x$.

From Corollary~\ref{corollary 1} we can see that every geodesic path from $x$ to $y$ is given by applying some sequence of the maps $\alpha_y$ and $\beta_y$ successively to the vertex $x$. In fact,  we will now show that applying the map $\alpha_y$ repeatedly gives a geodesic path from $x$ to $y$. This follows from the next lemma.

\begin{lemma}\label{lemma 3}
Suppose that $x$ and $y$ are distinct vertices of $\mathcal{F}_q$. Then
\[
d_q(\alpha_y(x),y)=d_q(x,y)-1.
\]
\end{lemma}
\begin{proof}
Let $n=d_q(x,y)$. The result is immediate if $n=1$, so let us assume that $n\geq 2$. Let $\gamma=\langle v_0,\dots,v_n\rangle$ be any geodesic path from $x$ to $y$, where $v_0=x$ and $v_n=y$. By Corollary~\ref{corollary 1} we know that $v_1$ is equal to either $\alpha_y(x)$ or $\beta_y(x)$. If the former is true, then our result is proved, so let us suppose instead that $v_1=\beta_y(x)$, in which case $d_q(\beta_y(x),y)=n-1$. We split our argument into two cases depending on whether $q$ is even or odd.

If $q$ is even then there is a vertex $w$ opposite $x$ on the face $P_y(x)$. The vertex $y$ lies in the opposite component of $\mathbb{R}_\infty\setminus\{x,w\}$ to $\beta_y(x)$ (or possibly
$y=w$) so Lemma~\ref{lemma 6} implies that the subpath $\langle v_1,\dots,v_n\rangle$ of $\gamma$ must pass through either $x$ or $w$. However, it cannot pass through $x$, as that is the initial vertex of $\gamma$, so it must pass through $w$. Corollary~\ref{corollary 3} tells us that $d_q(w,\alpha_y(x))=d_q(w,\beta_y(x))$, so we see that
\[
d_q(\alpha_y(x),y)=d_q(\beta_y(x),y)=n-1.
\]

Suppose now that $q$ is odd. Let $e$ be the edge of $P_y(x)$ opposite $x$, and let $Q$ be the other face of $\mathcal{F}_q$ incident to $e$. Let $w$ be the vertex opposite $x$ in the $2q$-gon formed by joining $P_y(x)$ and $Q$, as shown in Figure~\ref{figure 5}(b). We also define $a$ and $b$ to be the vertices incident to $e$, where $a$ lies in the same component of $\mathbb{R}_\infty\setminus\{x,w\}$ as $\alpha_y(x)$. The vertex $y$ either lies in the opposite component of $\mathbb{R}_\infty\setminus\{x,a\}$ to $\beta_y(x)$ or else it lies in the opposite component of $\mathbb{R}_\infty\setminus\{a,w\}$ to $\beta_y(x)$ (or possibly $y$ equals $a$ or $w$).  Lemma~\ref{lemma 6} implies that the subpath $\langle v_1,\dots,v_n\rangle$ of $\gamma$ must pass through either $x$,  $a$, or $w$ (and it cannot pass through $x$ as that is the initial vertex of $\gamma$). Since $a$ and $w$ are each at least as close to $\alpha_y(x)$ as they are to $\beta_y(x)$, we see once again that
\[
d_q(\alpha_y(x),y)=d_q(\beta_y(x),y)=n-1.\qedhere
\]
\end{proof}

The next corollary is an immediate consequence of this lemma.

\begin{corollary}\label{corollary 2}
Let $x$ and $y$ be distinct vertices of $\mathcal{F}_q$. Then there is a positive integer $m$ such that ${\alpha_y}^m(x)=y$, and $\langle x,\alpha_y(x),{\alpha_y}^2(x),\dots,{\alpha_y}^m(x)\rangle$ is a geodesic path from $x$ to $y$.
\end{corollary}

Corollary~\ref{corollary 2} tells us that to construct a geodesic path from $x$ to $y$, we can apply the map $\alpha_y$  repeatedly. Next we will show that when $x=\infty$, this geodesic path is the same as the path that arises from applying the nearest-integer algorithm. First we must describe that algorithm. Our description is essentially equivalent to Rosen's \cite{Ro1954}, but couched in the language of this paper.

We will show how to apply the nearest-integer algorithm to a real number $y$ to give a Rosen continued fraction in an inductive manner. For now we assume that $y$ is a vertex of $\mathcal{F}_q$, so that the continued fraction is finite. Later we will discuss the same algorithm when $y$ is not a vertex of $\mathcal{F}_q$. We use the notation $\Vert x\Vert_q$ to denote the nearest-integer multiple of $\lambda_q$ to $x$.

Let $b_1\lambda_q=\Vert y\Vert_q$, where $b_1\in \mathbb{Z}$. If $y$ lies half way between two integer multiples of $\lambda_q$, then we choose $b_1$ such that $b_1\lambda_q$ is the lesser of the two integer multiples of $\lambda_q$. Define $s_1(z)=b_1\lambda_q-1/z$.

Suppose now that we have constructed a sequence of integers $b_1,\dots,b_k$ and a corresponding sequence of maps $s_i(z)=b_i\lambda_q-1/z$, for $i=1,\dots,k$. We define $b_{k+1}\lambda_q=\Vert s_k^{-1}\dotsb s_1^{-1}(y)\Vert_q$, where $b_{k+1}\in\mathbb{Z}$, provided $s_k^{-1}\dotsb s_1^{-1}(y)$ is not $\infty$  (and as before we choose $b_{k+1}$ to be the lesser of the two integers in the ambiguous case). We then define $s_{k+1}(z)=b_{k+1}\lambda_q-1/z$.

We will prove shortly that $s_m^{-1}\dotsb s_1^{-1}(y)=\infty$ for some positive integer $m$, and at this stage the algorithm terminates. The outcome is a sequence of integers $b_1,\dots,b_m$ and maps $s_1,\dots,s_m$. Since $y=s_1\dotsb s_m(\infty)$, we see that the continued fraction $[b_1,\dots,b_m]_q$ has value $y$; it is called the \emph{nearest-integer continued fraction expansion} of $y$.

Let us now see why the nearest-integer algorithm applied to a vertex $y$ is equivalent to iterating the map $\alpha_y$. The key to this equivalence is  the following lemma.

\begin{lemma}\label{lemma 5}
Suppose that $y$ is a vertex of $\mathcal{F}_q$ other than $\infty$. Then $\alpha_y(\infty)=b\lambda_q$, for some integer $b$, where $b\lambda_q=\Vert y\Vert_q$ (and if $y$ lies half way between two integer multiples of $\lambda_q$, then $b\lambda_q$ is the lesser of the two).
\end{lemma}
\begin{proof}
The result is immediate if $y$ is an integer multiple of $\lambda_q$ (a neighbour of $\infty$), so let us assume that this is not so. In that case $y$ lies between $b\lambda_q$ and $(b+1)\lambda_q$, for some integer $b$. Therefore $P_y(\infty)$ is the face of $\mathcal{F}_q$ that is incident to $\infty$, $b\lambda_q$, and $(b+1)\lambda_q$, which implies that $\alpha_y(\infty)$ is equal to either $b\lambda_q$ or $(b+1)\lambda_q$. Using the generators $\tau$ and $\rho$ we can calculate all the vertices of $P_y(\infty)$ explicitly. Other than $\infty$, they are given by
\[
b\lambda_q+\frac{\sin\left(\pi(j-1)/q\right)}{\sin\left(\pi j/q\right)},\quad \text{for $j=1,\dots,q-1$}.
\]
When $q$ is even, the vertex opposite $\infty$ in $P_y(\infty)$ is $w=(b+1/2)\lambda_q$. If $y=w$, then $\alpha_y(\infty)=b\lambda_q$, because $b\lambda_q,\infty,(b+1)\lambda_q$ lie in that order clockwise around $\mathbb{R}_\infty$. So in this case $\alpha_y(\infty)=b\lambda_q$. If $y\neq w$, then $\alpha_y(\infty)$ lies in the same component of $\mathbb{R}\setminus\{w\}$ as $y$, so again it is equal to $b\lambda_q$.

When $q$ is odd, the argument is similar, if slightly more involved: we construct the face $Q$ as we did in Figure~\ref{figure 5}(b), and determine that the vertex opposite $\infty$ in the $2q$-gon made up of $P_y(\infty)$ joined to $Q$ is again $w=(b+1/2)\lambda_q$. We then proceed as before; the details are omitted.
\end{proof}

\begin{theorem}\label{theorem 7}
Let $b_1,b_2,\dotsc$ be the sequence of integers and $s_1,s_2,\dotsc$ the sequence of maps that arise in applying the nearest-integer algorithm to a real number $y$ that is a vertex of $\mathcal{F}_q$. Suppose that $d_q(\infty,y)=m$. Then $s_1\dotsb s_k(\infty)={\alpha_y}^k(\infty)$ for $k=1,\dots,m$. In particular, the nearest-integer algorithm terminates.
\end{theorem}
\begin{proof}
We proceed by induction on $k$. First, $s_1(\infty)=b_1\lambda_q=\alpha_y(\infty)$, using Lemma~\ref{lemma 5} for the second equality. Now suppose that $s_1\dotsb s_k(\infty)={\alpha_y}^k(\infty)$, where $k<m$. Then
\[
{\alpha_y}^{k+1}(\infty)=\alpha_y(s_1\dotsb s_k(\infty))=s_1\dotsb s_k(\alpha_{s_k^{-1}\dotsb s_1^{-1}(y)}(\infty)),
\]
using invariance of $\alpha$ under the transformation $s_1\dotsb s_k$ (an element of $G_q$) for the second equality. The vertex $s_k^{-1}\dotsb s_1^{-1}(y)$ of $\mathcal{F}_q$ is not $\infty$, for if it were then $y=s_1\dotsb s_k(\infty)={\alpha_y}^k(\infty)$. By definition, then, $b_{k+1}\lambda_q=\Vert s_k^{-1}\dotsb s_1^{-1}(y)\Vert_q$. Therefore, using Lemma~\ref{lemma 5} again,
\[
\alpha_{s_k^{-1}\dotsb s_1^{-1}(y)}(\infty) = b_{k+1}\lambda_q = s_{k+1}(\infty).
\]
Hence ${\alpha_y}^{k+1}(\infty)=s_1\dotsb s_{k+1}(\infty)$, which completes the inductive step.
\end{proof}

We can now prove Theorem~\ref{theorem 1} for finite  Rosen continued fractions.

\begin{proof}[Proof of Theorem~\ref{theorem 1}: finite case]
Let $y$ be a vertex of $\mathcal{F}_q$. Corollary~\ref{corollary 2} tells us that there is a positive integer $m$ such that ${\alpha_y}^m(\infty)=y$, and $\langle \infty, \alpha_y(\infty),\alpha^2_y(\infty),\dots,{\alpha_y}^m(\infty)\rangle$ is a geodesic path. Theorem~\ref{theorem 7} says that this path is the same as the path of convergents from $\infty$ to $y$ given by the nearest-integer continued fraction expansion of $y$. Therefore this expansion is a  geodesic Rosen continued fraction expansion.
\end{proof}

\section{Equivalent paths}\label{section 4}

In \cite{Ro1954}, Rosen described a sequence of operations that can be used to transform any Rosen continued fraction with value $y$ to the nearest-integer expansion of $y$. Here we explain informally how this process can be illuminated using paths in $\mathcal{F}_q$, without engaging with the details of Rosen's arguments. We will use only elementary graph theory.


To explain our method, let us begin with a finite connected plane graph $X$ and a path $\gamma$ in $X$. We define two elementary operations that can be applied to $\gamma$. The first is to either insert or remove a subpath of length two that proceeds from one vertex to a neighbouring vertex and immediately back again. The second is to either insert or remove a subpath that consists of a full circuit of the boundary of a face. These two operations preserve the start and end points of $\gamma$. They are illustrated in Figure~\ref{figure 9}.

\begin{figure}[ht]
\centering
\includegraphics{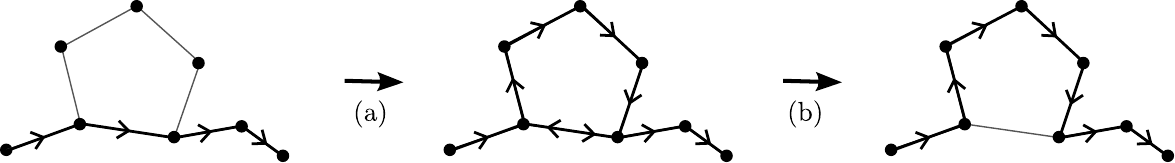}
\caption{(a) A clockwise circuit of the boundary of a face is inserted into the path (b) A subpath that proceeds from one vertex to a neighbouring vertex and immediately back again is removed from the path}
\label{figure 9}
\end{figure}

If we can transform $\gamma$ to another path $\gamma'$ by a finite sequence of these elementary operations, then we say that $\gamma$ and $\gamma'$ are \emph{homotopic}. There is a theory of homotopic paths for  more general 2-complexes than just finite connected plane graphs; see, for example, \cite[Section~1.2]{CoGrKuZi1998}. Using this concept of homotopy, one can define the fundamental group of a 2-complex in the obvious way, and, as you might expect, the fundamental group of a finite connected plane graph is trivial (see \cite[Corollary~1.2.14]{CoGrKuZi1998}). It follows that any two paths in a finite connected plane graph that start at the same vertex and finish at the same vertex are homotopic.

We can define homotopy for finite paths in the infinite graph $\mathcal{F}_q$ in exactly the same way as we have done for finite connected plane graphs, and we obtain the same conclusion about homotopic paths.

\begin{theorem}\label{theorem 5}
Any two paths in $\mathcal{F}_q$ that start at the same vertex and finish at the same vertex are homotopic.
\end{theorem}
\begin{proof}
Suppose that we start with a single face of $\mathcal{F}_q$, then adjoin all neighbouring faces, then adjoin all neighbouring faces of all those faces, and so forth. This yields a sequence of finite connected plane subgraphs of $\mathcal{F}_q$. We can choose a subgraph sufficiently far along the sequence that it contains all the vertices of the two paths. The two paths are homotopic in this subgraph, so they are homotopic in $\mathcal{F}_q$.
\end{proof}

Theorem~\ref{theorem 5} shows that any two paths from $\infty$ to a vertex $y$ of $\mathcal{F}_q$ are homotopic, so in particular, any path from $\infty$ to $y$ is homotopic to the path that we obtain by applying the nearest-integer algorithm to $y$. Using Lemma~\ref{lemma 7} we can reinterpret the two elementary operations as transformations of the Rosen continued fractions corresponding to the paths. This is illustrated by Figure~\ref{figure 13}. The first elementary operation corresponds to inserting or removing a $0$ coefficient in the continued fraction, and the second operation corresponds to inserting or removing  $q-1$ consecutive coefficients of value $1$ or $q-1$ consecutive coefficients of value $-1$ from the continued fraction. (Both  operations also impact on the neighbouring coefficients in the continued fraction; we will not go into this.) These two operations are essentially the same operations that Rosen uses in \cite{Ro1954}. They can be seen as applications of the two relations $\sigma^2=I$ and $\rho^q=I$ satisfied by the generators $\sigma$ and $\rho$ of $G_q$.

\begin{figure}[ht]
\centering
\includegraphics{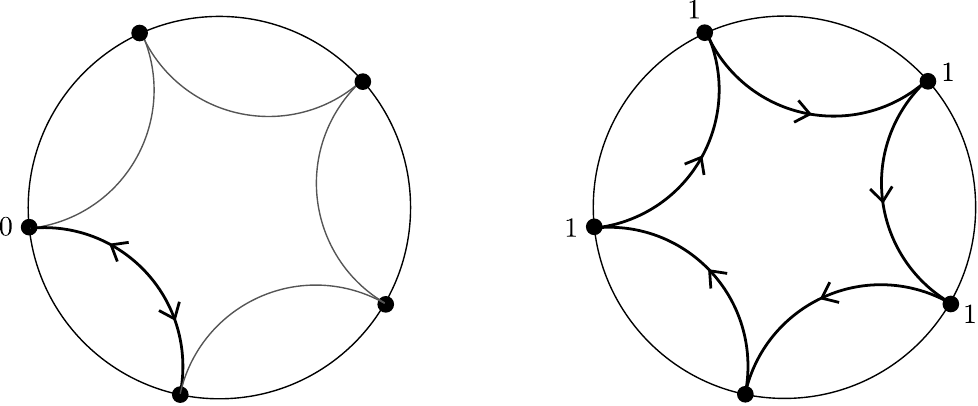}
\caption{Two paths in $\mathcal{F}_q$ labelled with coefficients of the corresponding continued fractions}
\label{figure 13}
\end{figure}

\section{Chains of $q$-gons}\label{section 5}

Let us now properly define the chain of $q$-gons between two non-adjacent vertices $x$ and $y$ of $\mathcal{F}_q$ that has so far only been introduced informally in the introduction. Consider a collection of Euclidean $q$-gons $P_1,\dots,P_n$ in the plane  such that $P_{i-1}$ and $P_i$ have a common edge for $i=2,\dots,n$ but otherwise the $q$-gons (including their interiors) do not overlap one another. Together these $q$-gons give rise to a connected, finite plane graph called a \emph{$q$-chain} whose vertices and edges are those of the constituent $q$-gons. For example, a $5$-chain is shown in Figure~\ref{figure 6}. We also refer to plane graphs that are topologically equivalent to $q$-chains as $q$-chains. 

Next we describe a process for constructing a $q$-chain $P_1,\dots,P_n$ consisting of faces of $\mathcal{F}_q$ such that $x$ is a vertex of $P_1$ and $y$ is a vertex of $P_n$. First, let $P_1$ be the face $P_y(x)$ (which was defined after  Lemma~\ref{lemma 2}). If $y$ is a vertex of $P_1$, then the construction terminates. Otherwise, there are two adjacent vertices $a_1$ and $b_1$ of $P_1$ such that $y$ belongs to the component of $\mathbb{R}_\infty\setminus\{a_1,b_1\}$ that contains no other vertices of $P_1$. Define $P_2$ to be the face of $\mathcal{F}_q$ other than $P_1$ that is also incident to the edge $\{a_1,b_1\}$. If $y$ is a vertex of $P_2$, then the construction terminates. Otherwise, there are two adjacent vertices $a_2$ and $b_2$ of $P_2$ such that $y$ belongs to the component of $\mathbb{R}_\infty\setminus\{a_2,b_2\}$ that contains no other vertices of $P_2$. We then define $P_3$ to be the face of $\mathcal{F}_q$ other than $P_2$ that is incident to $\{a_2,b_2\}$, and the procedure continues in this fashion. The resulting sequence of $q$-gons is uniquely defined by this process, because there is only one choice for each pair $\{a_i,b_i\}$. The first few $q$-gons in such a sequence are shown in Figure~\ref{figure 17}.

\begin{figure}[ht]
\centering
\includegraphics{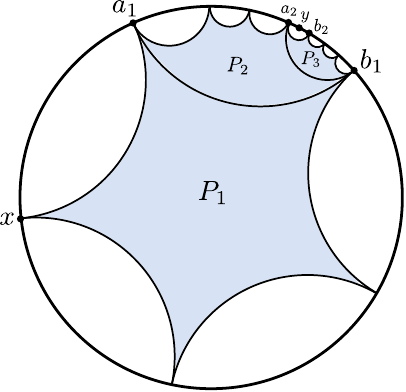}
\caption{The first three $q$-gons in a $q$-chain}
\label{figure 17}
\end{figure}

We must show that the procedure terminates. To this end, let us first show that a vertex $v$ of $\mathcal{F}_q$ can be incident to only finitely many consecutive faces of $P_1,P_2,\dotsc$. By applying a suitable element of $G_q$, we see that it suffices to prove this when $v=\infty$. The faces of $\mathcal{F}_q$ incident to $\infty$ are just the translates by iterates of $\tau$ of the fundamental domain $E$ of $\Gamma_q$, defined in the introduction, and it is straightforward to check that only finitely many of these $q$-gons can appear consecutively in the sequence $P_1,P_2,\dotsc$.

We deduce that there is a sequence of positive integers $n_1<n_2<\dotsb$ such that each pair of edges $\{a_{n_i},b_{n_i}\}$ and $\{a_{n_{i+1}},b_{n_{i+1}}\}$ do not have a common vertex. Let $d_i$ be the distance in the graph metric from $y$ to $a_{n_i}$ or $b_{n_i}$, whichever is nearest. It can easily be checked that $x$ and $y$ lie in distinct components of $\mathbb{R}_\infty\setminus\{a_j,b_j\}$, for each edge $\{a_j,b_j\}$, so Lemma~\ref{lemma 6} tells us that any path from $x$ to $y$ must pass through one of $a_j$ or $b_j$. It follows that $d_1,d_2,\dotsc$ is a decreasing sequence of positive integers, which must eventually terminate. Therefore the sequence $P_1,P_2,\dotsc$ has a final member $P_n$, which is incident to $y$.

The resulting sequence $P_1,\dots,P_n$ is a $q$-chain that is a subgraph of $\mathcal{F}_q$, which we call the \emph{$q$-chain from $x$ to $y$}. Figure~\ref{figure 17} shows that using the disc model of the hyperbolic plane, only a few faces from a $q$-chain are large enough (in Euclidean terms) that we can see them. Instead we usually draw $q$-chains using Euclidean polygons, as we did in Figure~\ref{figure 6}.

\begin{lemma}\label{lemma 4}
The $y$-parents of any vertex in the $q$-chain from $x$ to $y$ also belong to the $q$-chain.
\end{lemma}
\begin{proof}
Let us denote the chain by $P_1,\dots,P_n$. Choose a vertex $z$ of this $q$-chain other than $y$, and let $m$ be the largest integer such that $z$ is a vertex of $P_m$. Define $u$ and $v$ to be the vertices of $P_m$ adjacent to $z$. The point $y$ cannot lie in the component of $\mathbb{R}_\infty\setminus\{u,v\}$ that contains $z$, for if it did then $P_{m+1}$ would contain one of the edges $\{z,u\}$ or $\{z,v\}$, in which case $z$ would be a vertex of $P_{m+1}$. Therefore either $y$ is equal to $u$ or $v$, so that  $y$ is the single $y$-parent of $z$, or  otherwise $y$ belongs to the component of $\mathbb{R}_\infty\setminus\{u,v\}$ that does not contain $z$. Then, by definition (see Lemma~\ref{lemma 2}), $u$ and $v$ are the parents of $z$, so in particular they belong to the $q$-chain.
\end{proof}

There is an important corollary to this lemma.

\begin{theorem}\label{theorem 9}
Any geodesic path from a vertex $x$ to another vertex $y$ in $\mathcal{F}_q$ is contained in the $q$-chain from $x$ to $y$.
\end{theorem}
\begin{proof}
This is an immediate consequence of Corollary~\ref{corollary 1} and Lemma~\ref{lemma 4}.
\end{proof}

Theorem~\ref{theorem 9} tells us that to understand geodesic paths in Farey graphs, it suffices to understand geodesic paths in $q$-chains. This is a significant reduction because $q$-chains are simple, finite plane graphs. Later we use $q$-chains to prove Theorems~\ref{theorem 2} and \ref{theorem 3}.

Let us define a binary function $D$ on the vertices of $\mathcal{F}_q$, which we use in the next section, as follows. If $x$ and $y$ are equal, then $D(x,y)=0$, and if they are adjacent then $D(x,y)=1$. Otherwise, $D(x,y)$ is the number of $q$-gons in the $q$-chain from $x$ to $y$. The function $D$ was mentioned already, in the introduction. It is closely related to the graph metric on the dual graph of $\mathcal{F}_q$. However, although $D$ is symmetric, it does not satisfy the triangle inequality (we omit proofs of these two facts as we do not need them).

\section{Proof of Theorem~\ref{theorem 2}}\label{section 6}

We use the following notation in the proof of Theorem~\ref{theorem 2}. As in the previous section, we let $P_1,\dots,P_n$ be the $q$-chain between non-adjacent vertices $x$ and $y$ of $\mathcal{F}_q$, and let $\{a_i,b_i\}$ be the edge between $P_i$ and $P_{i+1}$ for $i=1,\dots,n-1$. We choose $a_i$ and $b_i$ such that $a_i$, $a_{i+1}$, $b_{i+1}$, and $b_i$ occur in that order clockwise around $P_{i+1}$. It is convenient to also define $a_0=b_0=x$ and $a_n=b_n=y$. We define $\mu_i$ to be the path on $P_i$ that travels clockwise from $a_{i-1}$ to $a_i$, and we define $\nu_i$ to be the path on $P_i$ that travels anticlockwise from $b_{i-1}$ to $b_i$. All this notation is illustrated in Figure~\ref{figure 14} (in which the edge between $a_i$ and $b_i$ is labelled by its length $1$).

\begin{figure}[ht]
\centering
\includegraphics{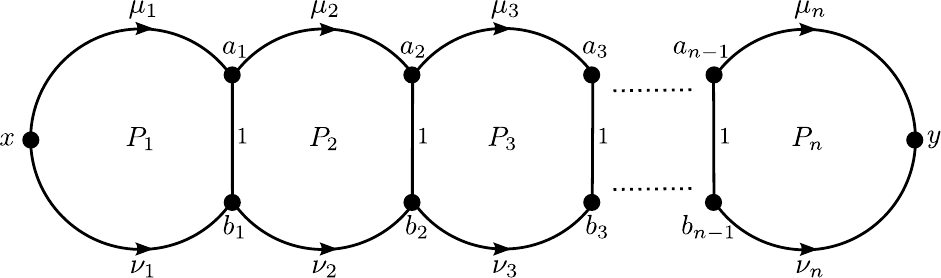}
\caption{A $q$-chain from $x$ to $y$}
\label{figure 14}
\end{figure}

We denote the length of a path $\gamma$ in $\mathcal{F}_q$ by $|\gamma|$. Note that $|\gamma|$ is the number of edges of $\gamma$, not the number of vertices (which is $|\gamma|+1$).

Let $N(x,y)$ denote the number of geodesic paths from $x$ to $y$. The following theorem, which gives bounds on $N$, is essentially equivalent to Theorem~\ref{theorem 2}, but is stated in slightly more generality. Recall that $F_n$ denotes the $n$th Fibonacci number, where $F_0=1$, $F_1=2$, $F_2$=3, $F_3=5$, and so forth.

\begin{theorem}\label{theorem 10}
Suppose that $x$ and $y$ are vertices of $\mathcal{F}_q$, and $D(x,y)=n$. Then $N(x,y)\leq F_{n}$. Furthermore, if $q$ is even then there are vertices $x$ and $y$ with $D(x,y)=n$ for which this bound can be attained.
\end{theorem}
\begin{proof}
We prove the inequality $N(x,y)\leq F_{n}$ by using induction on $n$. It is immediate if $n$ is $0$ or $1$. Suppose now that $n>1$, and assume that the inequality is true for all pairs of vertices $u$ and $v$ with $D(u,v)<n$. Choose two vertices $x$ and $y$ with $D(x,y)=n$, and let $P_1,\dots,P_n$ be the $q$-chain from $x$ to $y$, illustrated in Figure~\ref{figure 14}. Without loss of generality, we may assume that $|\mu_1|\leq |\nu_1|$. If $|\mu_1|<|\nu_1|-1$, then every geodesic path from $x$ to $y$ must pass along the path $\mu_1$. Since $D(a_1,y)=n-1$ we see by induction that
\[
N(x,y)=N(a_1,y)\leq F_{n-1} < F_{n}.
\]
The remaining possibility is that $|\nu_1|-1\leq |\mu_1| \leq |\nu_1|$. In this case, the set of geodesic paths from $x$ to $y$ can be partitioned into the set $A$ of those geodesic paths that travel along the path  $\mu_1$ and the set $B$ of those geodesic paths that travel along the path  $\nu_1$. A path in  $B$ cannot pass through $a_1$ (as well as $b_1$) because it is a geodesic path. Instead it must pass through $b_2$. Since $D(a_1,y)=n-1$ and $D(b_2,y)=n-2$ it follows by induction that
\[
N(x,y) = |A|+|B|\leq N(a_1,y)+N(b_2,y)\leq F_{n-1}+F_{n-2}=F_{n}.
\]
This completes the proof of the first assertion of the theorem.

To prove the second assertion of the theorem, that the bound $F_{n}$ can be attained when $q$ is even, we illustrate in Figure~\ref{figure 15} $q$-chains (with $q=2r$) that attain the bound, and highlight the $q=4$ case. The labels in Figure~\ref{figure 15} give the number of edges between pairs of vertices. The details to show that these examples do indeed attain the bound are omitted.
\end{proof}

\begin{figure}[ht]
\centering
\includegraphics{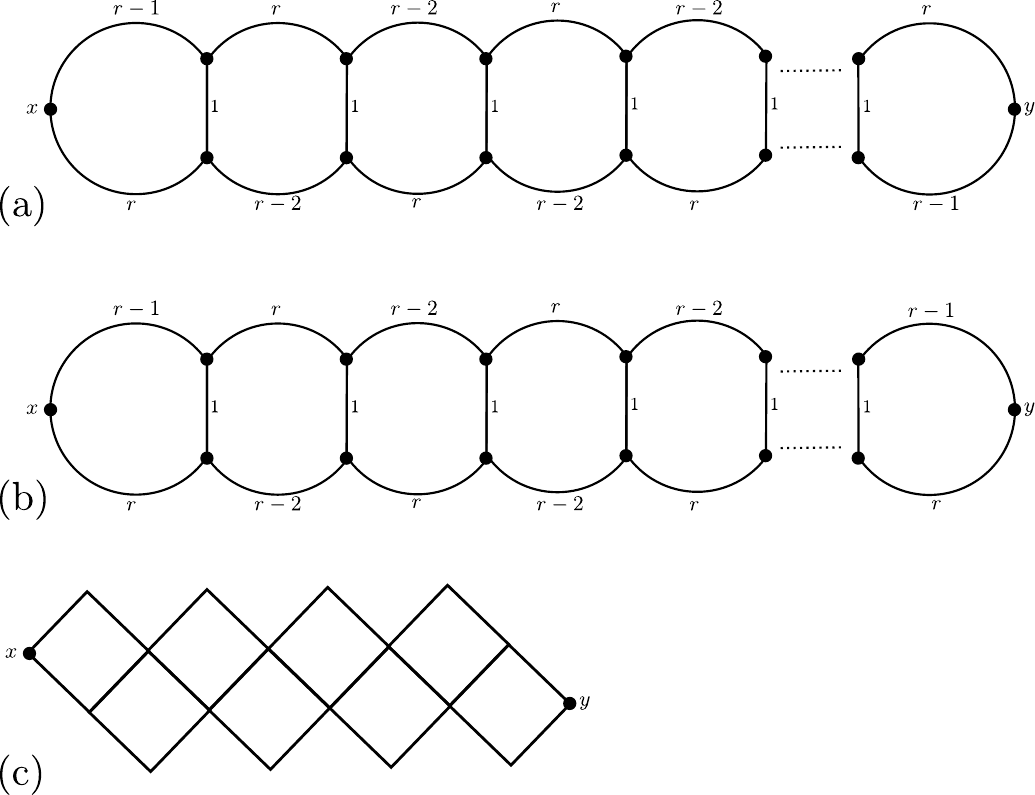}
\caption{(a) A $2r$-chain for which $N(x,y)=F_{n}$ ($n$ even) (b) A $2r$-chain for which $N(x,y)=F_{n}$ ($n$ odd)  (c) A $4$-chain for which $N(x,y)=F_{n}$ }
\label{figure 15}
\end{figure}

Theorem~\ref{theorem 10} does not give the best possible bounds for $N$ when $q$ is odd. Using a similar but more elaborate proof to that of Theorem~\ref{theorem 10}, one can show that, when $n>1$,
\[
N(x,y) \leq
\begin{cases}
 F_{n/2}, & \text{$n$ even},\\
2F_{(n-3)/2}, & \text{$n$ odd}.
\end{cases}
\]
This is the best possible bound, in the sense that, for each integer $n>1$, there are vertices $x$ and $y$ of $\mathcal{F}_q$ with $D(x,y)=n$ for which $N(x,y)$ achieves the bound. (There is one exception to this: when $q=3$ and $n$ is odd, there is a better bound $N(x,y)\leq F_{(n-1)/2}$.) Examples for which this bound is attained (when $q>3$) are shown in Figure~\ref{figure 16}.

\begin{figure}[ht]
\centering
\includegraphics{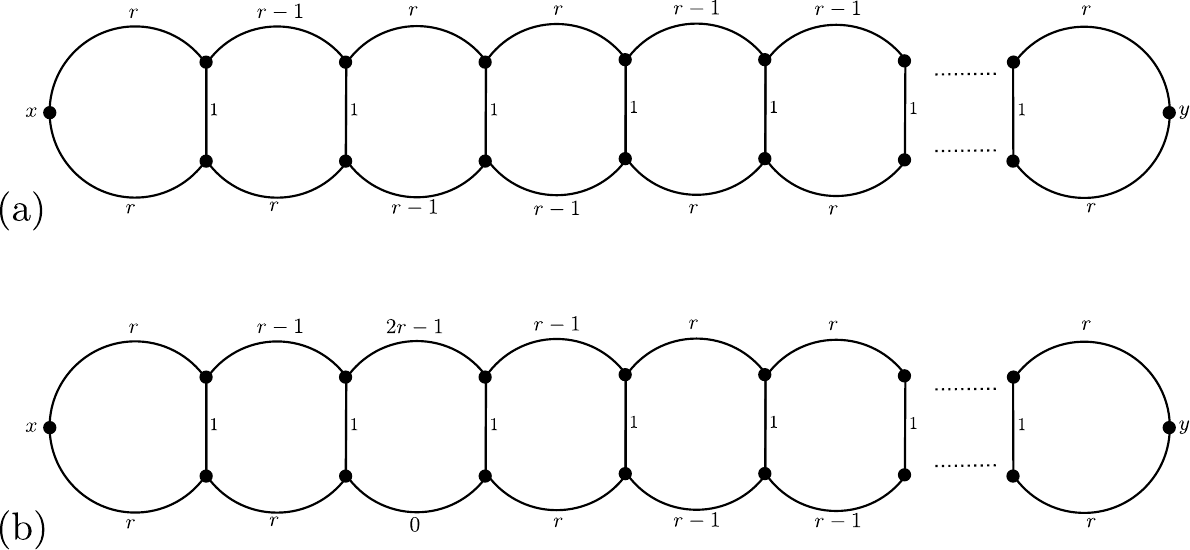}
\caption{(a) A $(2r+1)$-chain for which $N(x,y)= F_{n/2}$ ($n$ even) (b) A $(2r+1)$-chain for which $N(x,y)=F_{(n-3)/2}$  ($n$ odd)}
\label{figure 16}
\end{figure}

Theorem~\ref{theorem 10} gives an upper bound for the function $N$; the lower bound for $N$, no matter the value of $D(x,y)$, is $1$. For example, let $x=\infty$ and $y=[0,n]_q$. Then there is a unique geodesic path between $x$ and $y$, namely $\langle x,0,y\rangle$, and $D(x,y)=n$.

\section{Proof of Theorem~\ref{theorem 3}: part I}\label{section 7a}

Theorem~\ref{theorem 3} gives necessary and sufficient conditions for $[b_1,\dots,b_n]_q$ to be a geodesic Rosen continued fraction when $q$ is even. We now state a theorem of the same type when $q$ is odd, and greater than 3. The version of Theorem~\ref{theorem 12} when $q=3$ is a little different, and has been established already, in \cite[Theorem~1.3]{BeHoSh2012}.

\begin{theorem}\label{theorem 12}
Suppose that  $q=2r+1$, where $r\geq 2$. The continued fraction $[b_1,\dots,b_n]_q$ is a geodesic Rosen continued fraction if and only if the sequence $b_2,\dots,b_n$ has no terms equal to $0$ and contains no subsequence of consecutive terms either of the form $\pm  1^{[r]}$ or of the form
\[
\pm (1^{[d_1]},2,1^{[d_2]},2,\dots ,1^{[d_k]}),
\]
where $k$ is odd, at least $3$, and $d_1,\dots,d_k$ is the sequence
\[
r-1,\underbrace{r-1,r-2,r-1,r-2,\dots,r-1}_{\text{alternating $r-1$ and $r-2$}},r-1.
\]
\end{theorem}

For example, if $q=5$ (that is, $r=2$) then  the sequences $d_1,\dots,d_k$ described above are of the form
\[
1,1,1, \qquad 1,1,0,1,1,\qquad 1,1,0,1,0,1,1,
\]
and so forth.

In geometric terms, Theorem~\ref{theorem 12} says that the path of convergents of a Rosen continued fraction is a geodesic path unless it contains a subpath of the type shown in Figure~\ref{figure 23} (when $q=5$).

\begin{figure}[ht]
\centering
\includegraphics{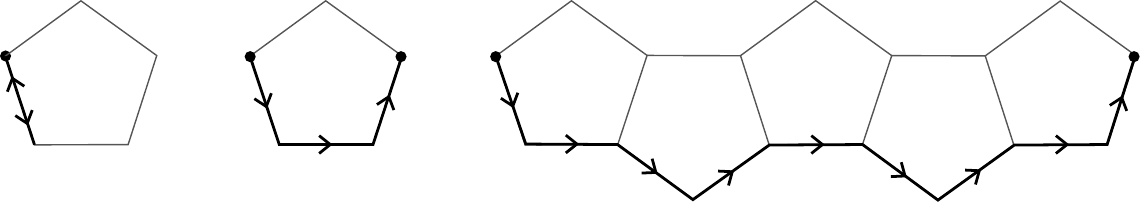}
\caption{A path in $\mathcal{F}_5$ that is not a geodesic path must contain a subpath of type similar to one of these.}
\label{figure 23}
\end{figure}

In this section we prove one of the implications from each of Theorems~\ref{theorem 3} and \ref{theorem 12} (the implication that says that if $b_2,\dots,b_n$ contains a subsequence of a certain type, then $[b_1,\dots,b_n]_q$ is not a geodesic Rosen continued fraction), leaving the other implications for the next section.

This part of the proofs relies on particular relations from the group $G_q$. Remember that $\sigma(z)=-1/z$, $\tau(z)=z+\lambda_q$, and $\kappa(z)=-\overline{z}$. For each integer $b$, let
\[
T_b(z)=b\lambda_q-\frac{1}{z}.
\]
That is, $T_b=\tau^b\sigma$. Observe that $\kappa T_b\kappa = T_{-b}$. This has the following useful consequence.

\begin{lemma}\label{lemma 18}
The continued fraction $[b_1,\dots,b_n]_q$ is a geodesic Rosen continued fraction if and only if $[-b_1,\dots,-b_n]_q$ is a geodesic Rosen continued fraction.
\end{lemma}
\begin{proof}
Let $\gamma=\langle \infty, v_1,\dots,v_n\rangle$ be the path of convergents of $[b_1,\dots,b_n]_q$ and let $\delta=\langle \infty, w_1,\dots,w_n\rangle$ be the path of convergents of $[-b_1,\dots,-b_n]_q$. Then
\begin{align*}
w_i &=T_{-b_1}\dotsb T_{-b_i}(\infty)\\
&=(\kappa T_{b_1}\kappa)\dotsb (\kappa T_{b_n}\kappa)(\infty)\\
&=\kappa T_{b_1}\dotsb T_{b_n}(\infty)\\
&= \kappa(v_i).
\end{align*}
Therefore $\delta$ is the image of $\gamma$ under the automorphism $\kappa$ of $\mathcal{F}_q$. It follows that $\gamma$ is a geodesic path if and only if $\delta$ is a geodesic path.
\end{proof}

The next two lemmas contain useful identities. Recall that $I$ is the identity element of $G_q$.

\begin{lemma}\label{lemma 17}
Let $q=2r$. Then ${T_1}^r=\sigma\tau^{-1}\sigma {T_{-1}}^{r-2}\tau^{-1}$.
\end{lemma}
\begin{proof}
Since $\sigma^2=I$, it is straightforward to check that the identity is equivalent to the relation $(\tau\sigma)^{2r}=I$.
\end{proof}

\begin{lemma}\label{lemma 16}
Let $q=2r$. Then for each integer $k=0,1,2,\dotsc$,
\[
{T_1}^{r-1}(T_2{T_1}^{r-2})^kT_2{T_1}^{r-1}=\sigma {\tau}^{-1}\sigma ({T_{-1}}^{r-2}\, T_{-2})^{k+1}{T_{-1}}^{r-2}\tau^{-1}.
\]
\end{lemma}
\begin{proof}
Using the relation $(\tau\sigma)^{2r}=I$, we can check that
\[
T_2{T_1}^{r-2} = (\tau\sigma\tau^{-1}\sigma)({T_{-1}}^{r-2}\, T_{-2})(\tau\sigma\tau^{-1}\sigma)^{-1}.
\]
Therefore
\begin{align*}
{T_1}^{r-1}(T_2{T_1}^{r-2})^kT_2{T_1}^{r-1} &= T_1{T_2}^{-1}(T_2{T_1}^{r-2})^{k+2}T_1\\
&= \tau^{-1}(T_2{T_1}^{r-2})^{k+2}\tau\sigma\\
&= \tau^{-1}(\tau\sigma\tau^{-1}\sigma)({T_{-1}}^{r-2}\, T_{-2})^{k+2}(\tau\sigma\tau^{-1}\sigma)^{-1}\tau\sigma\\
&= (\sigma\tau^{-1}\sigma)({T_{-1}}^{r-2}\, T_{-2})^{k+1}{T_{-1}}^{r-2}T_{-2}\sigma \tau\\
&= \sigma {\tau}^{-1}\sigma ({T_{-1}}^{r-2}\, T_{-2})^{k+1}{T_{-1}}^{r-2}\tau^{-1}.\qedhere
\end{align*}
\end{proof}

The next lemma describes what happens to the path of  convergents when the continued fraction has a zero coefficient.

\begin{lemma}\label{lemma 14}
Let $[b_1,\dots,b_n]_q$ be a Rosen continued fraction with path of convergents $\langle \infty,v_1,\dots,v_n\rangle$. Then, for each integer $i$ with $2\leq i\leq n$, $b_i=0$ if and only if $v_{i-2}=v_i$.
\end{lemma}
\begin{proof}
Let $s_j(z)=b_j\lambda_q-1/z$, so that $v_j=s_1\dotsb s_j(\infty)$. Then $v_{i-2}=v_i$ if and only if  $s_{i-1}^{-1}(\infty)=s_{i}(\infty)$; that is, if and only if $b_i\lambda_q=0$. The result follows.
\end{proof}

We can now prove the first part of Theorem~\ref{theorem 3}.

\begin{proof}[Proof of Theorem~\ref{theorem 3}: part I]
Suppose that the sequence $b_2,\dots,b_n$ either (i) contains a $0$ term; (ii) contains a subsequence of consecutive terms of the form $\pm  1^{[r]}$; or (iii) contains a subsequence of consecutive terms of the form
\[
\pm (1^{[r-1]},2,1^{[r-2]},2,1^{[r-2]},\dots,1^{[r-2]},2,1^{[r-1]}).
\]
We must prove that $[b_1,\dots,b_n]_q$ is not a geodesic Rosen continued fraction. Lemma~\ref{lemma 18} tells us that we can switch $[b_1,\dots,b_n]_q$ for $[-b_1,\dots,-b_n]_q$ if necessary so that in cases (ii) and (iii)  we have the $+$ form of the subsequence (1s and 2s rather than $-1$s and $-2$s).

In case (i), we know from Lemma~\ref{lemma 14} that the path of convergents contains two equal terms, so it is not a geodesic path.

In case (ii), there is an integer $i\geq 1$ such that $b_{i+1}=\dotsb =b_{i+r}=1$. Consider the alternative continued fraction
\[
[b_1,\dots,b_{i-1},b_i-1,\underbrace{-1,-1,\dots,-1}_{\text{$r-2$ copies of $-1$}},b_{i+r+1}-1,b_{i+r+2},\dots,b_n]_q.
\]
When $i=n-r$, this expression becomes
\[
[b_1,\dots,b_{i-1},b_i-1,\underbrace{-1,-1,\dots,-1}_{\text{$r-2$ copies of $-1$}}]_q.
\]
The alternative continued fraction is shorter than $[b_1,\dots,b_n]_q$, and using Lemma~\ref{lemma 17} we can check that the two continued fractions have the same value:
\begin{align*}
T_{b_1}\dotsb T_{b_n}(\infty) &= T_{b_1}\dotsb T_{b_i}{T_1}^rT_{b_{i+r+1}}\dotsb T_{b_n}(\infty) \\
&= T_{b_1}\dotsb T_{b_i}\sigma\tau^{-1}\sigma {T_{-1}}^{r-2}\tau^{-1}T_{b_{i+r+1}}\dotsb T_{b_n}(\infty) \\
&= T_{b_1}\dotsb T_{b_{i-1}}T_{b_i-1} {T_{-1}}^{r-2}T_{b_{i+r+1}-1}T_{b_{i+r+2}}\dotsb T_{b_n}(\infty).
\end{align*}
Therefore $[b_1,\dots,b_n]_q$ is not a geodesic Rosen continued fraction.

In case (iii), there are integers $i$ and $j$ with $1\leq i<j\leq n$ such that
\[
b_{i+1},\dots,b_j=1^{[r-1]},2,1^{[r-2]},2,1^{[r-2]},\dots,1^{[r-2]},2,1^{[r-1]}.
\]
Let $b_{i+1}^*,\dots,b_{j-2}^*$ be the shorter sequence given by
\[
-b_{i+1}^*,\dots,-b_{j-2}^*=1^{[r-2]},2,1^{[r-2]},2,1^{[r-2]},\dots,1^{[r-2]},2,1^{[r-2]}.
\]
Consider the continued fraction
\[
[b_1,\dots,b_{i-1},b_i-1,b_{i+1}^*,\dots,b_{j-2}^*,b_{j+1}-1,b_j,\dots,b_n]_q
\]
(with the obvious interpretation when $j=n$). This is shorter than $[b_1,\dots,b_n]_q$, and using Lemma~\ref{lemma 16} you can check that the two continued fractions have the same value. So once again $[b_1,\dots,b_n]_q$ is not a geodesic Rosen continued fraction.
\end{proof}

The first part of the proof of Theorem~\ref{theorem 12} (which says that if $b_2,\dots,b_n$ contains a subsequence of one of the given types then $[b_1,\dots,b_n]_q$ is not a geodesic Rosen continued fraction) follows on exactly the same lines as the proof given above. The only significant difference is that instead of the identities in Lemmas~\ref{lemma 17} and \ref{lemma 16} we need ${T_1}^r=\sigma\tau^{-1}\sigma {T_{-1}}^{r-1}\tau^{-1}$ and
\begin{align*}
& {T_1}^{r-1}(T_2{T_1}^{r-1}T_2{T_1}^{r-2})^kT_2{T_1}^{r-1}T_2{T_1}^{r-1}\\
&=\sigma\tau^{-1}\sigma({T_{-1}}^{r-1}T_{-2}{T_{-1}}^{r-2}T_{-2})^{k+1}{T_{-1}}^{r-1}\tau^{-1},
\end{align*}
where $q=2r+1$. To prove the latter identity, it is helpful to first observe that
\[
T_2{T_1}^{r-1}T_2{T_1}^{r-2}=(\tau\sigma\tau^{-1}\sigma){T_{-1}}^{r-2}T_{-2}{T_{-1}}^{r-1}T_{-2}(\tau\sigma\tau^{-1}\sigma)^{-1};
\]
we omit the details.

\section{Proof of Theorem~\ref{theorem 3}: part II}\label{section 7b}

In this section we prove the more difficult parts of Theorems~\ref{theorem 3} and \ref{theorem 12}. Our method is thoroughly different to that of the previous section, and uses basic properties of $q$-chains. In fact, both parts of the two theorems could be proved using the techniques of this section, as many of the arguments we present can, with care, be reversed.

Let us start by introducing some new terminology for paths, which involves a concept that we met earlier. A path $\langle v_0,\dots,v_n\rangle$ in $\mathcal{F}_q$ is said to \emph{backtrack} if $v_i=v_{i+2}$ for some integer $i$ with $0\leq i\leq n-2$. That is, a path backtracks if it has a subpath that proceeds from one vertex to a neighbouring vertex and then immediately back again.

We define $P_1,\dots,P_m$ to be the $q$-chain from a vertex $x$ to a non-adjacent vertex $y$ in $\mathcal{F}_q$. Let $a_i$, $b_i$, $\mu_i$, and $\nu_i$ be the vertices and paths associated to this $q$-chain that were introduced at the start of Section~\ref{section 6}.  Given two paths $\gamma=\langle v_0,\dots,v_r\rangle$ and $\delta=\langle w_0,\dots,w_s\rangle$ in $\mathcal{F}_q$ such that $v_r=w_0$, we define $\gamma \delta$ to be the path $\langle v_0,\dots,v_r,w_1,\dots,w_s\rangle$. With this notation we can distinguish two particular paths from $x$ to $y$ in the $q$-chain, namely $\alpha=\mu_1  \mu_2 \dotsb \mu_m$ and $\beta=\nu_1  \nu_2 \dotsb \nu_m$. We refer to paths in $\mathcal{F}_q$ of this type as \emph{outer paths}. More specifically, paths of the same type as $\alpha$ are called \emph{clockwise outer paths} and paths of the same type as $\beta$ are called \emph{anticlockwise outer paths}.

The notation for the $q$-chain from $x$ to $y$ that was introduced in the previous paragraph will be retained for the rest of this section.

\begin{lemma}\label{lemma 9}
Let $x$ and $y$ be two non-adjacent vertices of $\mathcal{F}_q$. Suppose that $\chi$ is a path from $x$ to $y$ that lies in the $q$-chain from $x$ to $y$ and does not backtrack, and suppose that $\chi$ is not an outer path. Then $\chi$ intersects every path from $x$ to $y$ in a vertex other than $x$ or $y$.
\end{lemma}
\begin{proof}
As $\chi$ is not an outer path, it must traverse one of the edges $\{a_i,b_i\}$ for some integer $i$ with $1\leq i\leq m-1$. By Lemma~\ref{lemma 6}, every path from $x$ to $y$ contains one of the vertices $a_i$ or $b_i$, so $\chi$ intersects every path from $x$ to $y$ at one of these vertices.
\end{proof}

The outer paths $\alpha$ and $\beta$ may be equal in length, or one may be longer than the other. If one is longer than the other, then the longer one is called a \emph{circuitous path}. We define a \emph{minimal circuitous path} to be a circuitous path whose non-trivial subpaths are all geodesic paths. Notice that a circuitous path $\langle v_0,\dots,v_n\rangle$ is a minimal circuitous path if and only if both $\langle v_1,\dots,v_n\rangle$ and $\langle v_0,\dots,v_{n-1}\rangle$ are geodesic paths.

\begin{theorem}\label{theorem 16}
A path in $\mathcal{F}_q$ is a geodesic path if and only if it does not backtrack and it does not contain a  minimal circuitous subpath.
\end{theorem}
\begin{proof}
If a path backtracks or contains a circuitous subpath, then, by definition, it is not a geodesic path. To prove the converse, suppose that  $\gamma=\langle v_0,\dots,v_n\rangle$ is a path in $\mathcal{F}_q$ from $x$ to $y$ that is not a geodesic path. Suppose also that the path does not backtrack. We must prove that $\gamma$ contains a minimal circuitous subpath. To do this, we can, by restricting to a subpath of $\gamma$ if necessary,  assume that  every non-trivial subpath of $\gamma$ is a geodesic path. 

Let us first consider the cases in which $x$ and $y$ are either equal or adjacent. Elementary arguments show that, given the conditions just stated, the only possibility is that $\gamma=\langle x, v_1,v_2,y\rangle$, where $x$, $v_1$, $v_2$, and $y$ are distinct vertices, and $x$ and $y$ are adjacent. This can only happen if $q=4$ and $\gamma$ completes three sides of a face of $\mathcal{F}_q$, in which case $\gamma$ is a minimal circuitous path.

Suppose now that $x$ and $y$ are neither equal nor adjacent. We will show that $\gamma$ is contained within the $q$-chain from $x$ to $y$. Theorem~\ref{theorem 6} shows that any path from $x$ to $y$ must pass through one of the $y$-parents of $x$. Let $i$ be the smallest positive integer such that $v_i$ is one of the $y$-parents of $x$. Since $\langle x,v_1,\dots,v_i\rangle$ is a geodesic path, it must be that $i=1$, because otherwise $\langle x,v_i\rangle$ is a shorter path from $x$ to $v_i$.  Lemma~\ref{lemma 4} tells us that $v_1$ belongs to the $q$-chain from $x$ to $y$. Repeating this argument we see that all vertices $v_0,\dots,v_n$ belong to the $q$-chain from $x$ to $y$.

Now let $\delta$ be a geodesic path from $x$ to $y$, which, by Theorem~\ref{theorem 9}, also lies in the $q$-chain from $x$ to $y$. The paths $\gamma$ and $\delta$ can only intersect at the vertices $x$ and $y$, because if they intersect at some other vertex $z$, then one of the subpaths of $\gamma$  from $x$ to $z$ or from $z$ to $y$ is not a geodesic path. It follows from Lemma~\ref{lemma 9} that $\gamma$ and $\delta$ are both outer paths, and since $\gamma$ is not a geodesic path it must be a circuitous path. In fact $\gamma$ is a minimal circuitous path, as each of its non-trivial subpaths is a geodesic path.
\end{proof}

Recall the notation for the outer paths $\alpha$ and $\beta$ defined before Lemma~\ref{lemma 9}.

\begin{lemma}\label{lemma 12}
If $\alpha$ is a minimal circuitous path, then $\beta$ is the only geodesic path from $x$ to $y$.
\end{lemma}
\begin{proof}
Let $\delta$ be a geodesic path from $x$ to $y$. By Theorem~\ref{theorem 9}, this path is contained in the $q$-chain from $x$ to $y$. Suppose that it contains one of the vertices $a_i$, where $1\leq i\leq m-1$. Since $\mu_1\dotsb \mu_i$ is a \emph{geodesic} path from $x$ to $a_i$ (as it is a non-trivial subpath of $\alpha$) and $\mu_{i+1}\dotsb \mu_m$ is a geodesic path from $a_i$ to $y$, it follows that $\alpha$ is a geodesic path, which is a contradiction. Therefore $\delta$ does not contain any of the vertices $a_i$, so it must equal $\beta$.
\end{proof}

We can characterise the minimal circuitous paths precisely, and we do so in Lemma~\ref{lemma 10} and Theorem~\ref{theorem 17}, below.

\begin{lemma}\label{lemma 10}
Let $q$ be equal to either $2r$ or $2r+1$, where $r\geq 2$. If $\alpha$ is a minimal circuitous path on a $q$-chain of length $1$ (that is, $m=1$), then $|\mu_1|=r+1$.
\end{lemma}
\begin{proof}
Since $\alpha$ is a circuitous path, and $m=1$, we have $|\mu_1|>|\nu_1|$. Since $|\mu_1|+|\nu_1|\geq 2r$, we see that $|\mu_1|\geq r+1$. If $|\mu_1|>r+1$, then $\alpha$ contains a non-geodesic subpath (just remove the final vertex from $\mu_1$). Therefore $|\mu_1|=r+1$.
\end{proof}

The next two lemmas are needed to prove Theorem~\ref{theorem 17}.

\begin{lemma}\label{lemma 11}
Let $q$ be equal to either $2r$ or $2r+1$, where $r\geq 2$. If $\alpha$ is a minimal circuitous path on a $q$-chain of length at least $2$ (that is, $m\geq 2$), then $r-1 \leq |\mu_i| \leq r$ for $i=1,\dots,m$. Furthermore, $|\mu_1|=|\mu_m|=r$.
\end{lemma}
\begin{proof}
Suppose that $|\mu_i|\geq r+1$. Then $\mu_i$ is not a geodesic path, which is impossible, as it is a subpath of $\alpha$. Therefore $|\mu_i|\leq r$.

Suppose next that $|\mu_i|\leq r-2$, where $1<i<m$. Let $\delta$ be the path from $x$ to $y$ given by
\[
\delta = \nu_1\dotsb \nu_{i-1}\langle b_{i-1},a_{i-1}\rangle\mu_i\langle a_{i},b_{i}\rangle\nu_{i+1}\dotsb \nu_m.
\]
Since $|\mu_i|+|\nu_i|+2\geq 2r$, we see that
\[
|\beta|-|\delta| = |\nu_i|-|\mu_i|-2 \geq 2(r-2-|\mu_i|)\geq 0.
\]
Therefore $|\delta|\leq |\beta|$, so $\delta$ is a geodesic path from $x$ to $y$, which contradicts Lemma~\ref{lemma 12}. Therefore $|\mu_i|\geq r-1$ when $1<i<m$.

Finally, we prove that $|\mu_1|=r$; the proof that $|\mu_m|=r$ is similar and omitted. Suppose that $|\mu_1|\leq r-1$. Let $\delta$ be the path from $x$ to $y$ given by
\[
\delta = \mu_1\langle a_{1},b_{1}\rangle\nu_{2}\dotsb \nu_m.
\]
Since $|\mu_1|+|\nu_1|+1\geq 2r$, we see that
\[
|\beta|-|\delta| = |\nu_1|-|\mu_1|-1 \geq 2(r-1-|\mu_i|)\geq 0.
\]
Therefore $\delta$ is a geodesic path, which contradicts Lemma~\ref{lemma 12}. Therefore $|\mu_1|\geq r$. Since we have already proved that $|\mu_1|\leq r$, we conclude that $|\mu_1|=r$.
\end{proof}

\begin{lemma}\label{lemma 13}
Let $q=2r+1$, where $r\geq 2$. If $\alpha$ is a minimal circuitous path on a $q$-chain of length at least $3$ (that is, $m\geq 3$), then $|\mu_2|=|\mu_{m-2}|=r$. Also, for no integer $i$ with $1< i<m-1$ are $|\mu_i|$ and $|\mu_{i+1}|$ both equal to $r-1$.
\end{lemma}
\begin{proof}
Let us first prove that $|\mu_2|=r$. The proof that $|\mu_{m-2}|=r$ is similar and omitted. We know from Lemma~\ref{lemma 11} that $|\mu_2|$ is either $r-1$ or $r$. If $|\mu_2|=r-1$, then the path
\[
\delta = \mu_1\mu_2\langle a_{2},b_{2}\rangle\nu_{2}\dotsb \nu_m
\]
is a geodesic path from $x$ to $y$, which contradicts Lemma~\ref{lemma 12}. Therefore $|\mu_2|=r$.

For the second assertion of the lemma, suppose that $|\mu_i|=|\mu_{i+1}|=r-1$ for some integer $i$, where $1< i<m$. In this case, the path
\[
\delta = \nu_1\dotsb \nu_{i-1}\langle b_{i-1},a_{i-1}\rangle\mu_{i}\mu_{i+1}\langle a_{i+1},n_{i+1}\rangle\nu_{i+2}\dotsb \nu_m
\]
is a geodesic path from $x$ to $y$, which contradicts Lemma~\ref{lemma 12}.
\end{proof}

\begin{theorem}\label{theorem 17}
Suppose that $q\geq 4$. Let $\alpha$ be a clockwise outer path  on a $q$-chain of length at least $2$ (that is, $m\geq 2$). If $\alpha$ is a minimal circuitous path, then
\begin{equation}\label{equation 1}
|\mu_1|,\dots,|\mu_m| =
\begin{cases}
r,r-1,r-1,\dots,r-1,r \ (m\geq 2), & \text{if $q=2r$},\\
r,r,r-1,r,r-1,\dots,r-1,r,r \ (m\geq 3), & \text{if $q=2r+1$}.
\end{cases}
\end{equation}
\end{theorem}

The sequences $|\mu_1|,\dots,|\mu_m|$ described above are
\[
r,r,\qquad r,r-1,r, \qquad r, r-1,r-1,r,
\]
and so on, when $q=2r$, and
\[
r,r,r\qquad r,r,r-1,r,r \qquad r,r, r-1,r,r-1,r,r,
\]
and so on, when $q=2r+1$. There is a converse to Theorem~\ref{theorem 17}, which says that if $\alpha$ is an outer path and $|\mu_1|,\dots,|\mu_m|$ is one of these sequences, then $\alpha$ is a minimal circuitous path. We do not prove this converse result as we do not need it.

\begin{proof}
Suppose first that $\alpha$ is an outer path and $|\mu_1|,\dots,|\mu_m|$ takes one of the values given in \eqref{equation 1}; we do not yet assume that $\alpha$ is a minimal circuitous path. Let us prove that $\alpha$ is not a geodesic path.  To do this, we will show that $|\beta|<|\alpha|$. This is true when $q=2r$, because  $|\alpha|=m(r-1)+2$ and $|\alpha|+|\beta|=2m(r-1)+2$, so $|\beta|=m(r-1)$. It is also true when $q=2r+1$, because $|\alpha|=(2r-1)m/2+3/2$ and $|\alpha|+|\beta|=(2r-1)m+2$, so $|\beta|=(2r-1)m/2+1/2$.

Now let us assume that $\alpha$ is a minimal circuitous path. By Lemma~\ref{lemma 11}, $|\mu_i|$ is either $r-1$ or $r$ for $i=1,\dots,m$.

Suppose that $q=2r$. Lemma~\ref{lemma 11} tells us that $|\mu_1|=|\mu_m|=r$. Suppose that $|\mu_i|=r$ for some integer $i$ with $1<i<m$; in fact, let $i$ be the smallest such integer. Then $\mu_1\dotsb \mu_i$ is not a geodesic path, as we demonstrated at the start of this proof, which is a contradiction, as it is a subpath of $\mu_1\dotsb \mu_m$. Hence $|\mu_i|=r-1$ for $1<i<m$, as required.

Suppose now that $q=2r+1$. Lemmas~\ref{lemma 11} and \ref{lemma 13} tell us that $|\mu_1|=|\mu_2|=|\mu_{m-1}|=|\mu_m|=r$. The path $\mu_1\dotsb \mu_m$ is not circuitous when $m$ is $2$ or $4$ (because it is the same length as $\beta$ in each case). However, it is circuitous when $m=3$.

Let us assume then that $m\geq 5$. We know that $|\mu_3|=r-1$, because if $|\mu_3|=r$, then the subpath $\mu_1\mu_2\mu_3$ of $\alpha$ is not a geodesic path. The second assertion of Lemma~\ref{lemma 13} tells us that $|\mu_4|=r$. Next, if $|\mu_5|=r$, then, as we saw at the start of this proof, $\mu_1\dotsb \mu_5$ is not a geodesic path, which can only be so if $m=5$. Therefore $|\mu_5|=r-1$ when $m>5$. Arguing repeatedly in this fashion, we see that $|\mu_1|,\dots,|\mu_m| =r,r,r-1,r,r-1,\dots,r-1,r,r$, where $m$ is odd and at least~3.
\end{proof}

The next lemma allows us to move from paths to continued fractions.

\begin{lemma}\label{lemma 15}
Let $\langle v_0,v_1,\dots,v_n\rangle$, where $v_0=\infty$, be the path of convergents of a Rosen continued fraction $[b_1,\dots,b_n]_q$. Suppose that $\langle v_k,\dots,v_l\rangle$, where  $0\leq k<l\leq n$ and $k+2\leq l$, is a clockwise outer path $\alpha=\mu_1\dotsb \mu_m$  such that $|\mu_i|\geq 1$ for $i=1,\dots,m$. Then 
\[
b_{k+2},\dots,b_l = 1^{[|\mu_1|-1]},2,1^{[|\mu_2|-1]},2,\dots,2,1^{[|\mu_m|-1]}.
\]
\end{lemma}
\begin{proof}
Recall the function $\phi$ defined near the end of Section~\ref{section 2}. Lemma~\ref{lemma 7} tells us that $\phi(v_{i-1},v_{i},v_{i+1})=b_{i+1}$ for $i=k+1,\dots,l-1$, so we need only calculate these values of $\phi$ one by one. If $v_i$ is not one of the vertices $a_j$, then $v_{i-1}$, $v_i$, and $v_{i+1}$ lie in that order clockwise round a face of $\mathcal{F}_q$. Therefore $\phi(v_{i-1},v_{i},v_{i+1})=1$. If $v_i$ is one of the vertices $a_j$, then because $|\mu_{i}|$ and $|\mu_{i+1}|$ are both at least 1, $v_{i-1}$, $v_i$, and $v_{i+1}$ lie in that order clockwise around a $2q$-gon comprised of two adjacent faces of $\mathcal{F}_q$ that meet along an edge that has $v_i$ as a vertex. If we map $v_i$ to $\infty$ by an element of $G_q$, then we see that  $\phi(v_{i-1},v_{i},v_{i+1})=2$. Therefore $b_{k+2},\dots,b_l$ is of the given form.
\end{proof}

Finally, we are able to prove the second part of Theorem~\ref{theorem 3}.

\begin{proof}[Proof of Theorem~\ref{theorem 3}: part II]
Suppose that $[b_1,\dots,b_n]_q$ is not a geodesic Rosen continued fraction. We must prove that the sequence $b_2,\dots,b_n$ either (i) contains a $0$ term; (ii) contains a subsequence of consecutive terms of the form $\pm  1^{[r]}$; or (iii) contains a subsequence of consecutive terms of the form
\[
\pm (1^{[r-1]},2,1^{[r-2]},2,1^{[r-2]},\dots,1^{[r-2]},2,1^{[r-1]}).
\]

Since the path of convergents of $\gamma=\langle \infty, v_1,\dots,v_n\rangle$ is not a geodesic path, Theorem~\ref{theorem 16}  tells us that it either backtracks or contains a minimal circuitous path. If it backtracks, then, by Lemma~\ref{lemma 14}, one of the coefficients $b_2,\dots,b_n$ is $0$, which is case (i). Otherwise, $\gamma$ contains a minimal circuitous subpath. Suppose for the moment that this minimal circuitous subpath is a \emph{clockwise} outer path, namely $\alpha=\mu_1\dotsb \mu_m$. Lemma~\ref{lemma 10} and Theorem~\ref{theorem 17} tell us that either $m=1$ and $|\mu_1|=r+1$ or $m>1$ and
\[
|\mu_1|,\dots,|\mu_m| = r,r-1,r-1,\dots,r-1,r.
\]
It follows from Lemma~\ref{lemma 15} that there are integers $0\leq k<l\leq n$ with $k+2\leq l$ such that the sequence $b_{k+2},\dots,b_l$ is $1^{[r]}$, when $m=1$, or
\[
1^{[r-1]},2,1^{[r-2]},2,1^{[r-2]},\dots,1^{[r-2]},2,1^{[r-1]},
\]
when $m>1$. These are cases (ii) and (iii).

Earlier we assumed that $\gamma$ contained a minimal circuitous subpath that was a \emph{clockwise} outer path; suppose now that it is an \emph{anticlockwise} outer path $\beta$. Recall that the map $\kappa(z)=-\overline{z}$ is an anticonformal transformation of the upper half-plane that induces an automorphism of $\mathcal{F}_q$. As we saw in Lemma~\ref{lemma 18}, the path of convergents of the continued fraction $[-b_1,\dots,-b_n]_q$ is $\kappa(\gamma)$. The path $\kappa(\beta)$ is a minimal circuitous subpath of $\kappa(\gamma)$, but it is a clockwise outer path rather than an anticlockwise outer path, as $\kappa$ reverses the orientation of cycles in $\mathcal{F}_q$. Since, by Lemma~\ref{lemma 18}, $[-b_1,\dots,-b_n]_q$ is not a geodesic Rosen continued fraction, the argument from the previous paragraph shows that $-b_2,\dots,-b_n$ contains a subsequence of consecutive terms of the form $1^{[r]}$ or $1^{[r-1]},2,1^{[r-2]},\dots,1^{[r-2]},2,1^{[r-1]}$. Therefore we see once again that one of statements (ii) or (iii) holds.
\end{proof}

The proof of the second part of Theorem~\ref{theorem 12} mirrors the proof given above almost exactly, but with the sequence of $1$s and $2$s and the sequence of $r$s and $(r-1)$s suitably modified.

\section{Infinite Rosen continued fractions}\label{section 8}

So far, for simplicity, we have focussed on finite continued fractions; however, most of our theorems and techniques generalise in a straightforward fashion to infinite continued fractions. Here we briefly discuss the theory of infinite Rosen continued fractions, and in particular we prove Theorem~\ref{theorem 13}.

An \emph{infinite path} in $\mathcal{F}_q$ is a sequence of vertices $v_0,v_1,\dotsc$ such that $v_{i-1}\sim v_i$ for $i=1,2,\dotsc$. The convergents of an infinite Rosen continued fraction form an infinite path $\langle \infty,v_1,v_2,\dotsc\rangle$ in $\mathcal{F}_q$, and conversely each infinite path of this type is comprised of the convergents of a unique infinite Rosen continued fraction.

To prove Theorem~\ref{theorem 13}, we use the following lemma.

\begin{lemma}\label{lemma 8}
Suppose that $x$ and $y$ are distinct elements of $\mathbb{R}_\infty$ that are not adjacent vertices of $\mathcal{F}_q$. Then there are two vertices $u$ and $v$ of some face of $\mathcal{F}_q$ such that $x$ and $y$ lie in distinct components of $\mathbb{R}_\infty\setminus\{u,v\}$.
\end{lemma}
\begin{proof}
If $x$ and $y$ are both vertices of $\mathcal{F}_q$, then this assertion follows from Lemma~\ref{lemma 2}. If one of $x$ and $y$ is a vertex of $\mathcal{F}_q$ (say $x$) and the other is not, then after applying an element of $G_q$ we may assume that $x=\infty$, in which case we can choose $u$ and $v$  to be the integer multiples of $\lambda_q$ that lie either side of $y$ on the real line. Suppose finally that neither $x$ nor $y$ are vertices of $\mathcal{F}_q$. The hyperbolic line from $x$ to $y$ must intersect an edge of $\mathcal{F}_q$ (else this hyperbolic line disconnects $\mathcal{F}_q$) and the end points of this edge are the required vertices $u$ and $v$.
\end{proof}

\begin{proof}[Proof of Theorem~\ref{theorem 13}]
We prove the contrapositive of Theorem~\ref{theorem 13}. Suppose that the sequence of convergents $v_1,v_2,\dotsc$ of an infinite Rosen continued fraction diverges. Then this sequence has two distinct limit points, $x$ and $y$. Assume for the moment that $x$ and $y$ are not adjacent vertices of $\mathcal{F}_q$. Then Lemma~\ref{lemma 8} tells us that there are two vertices $u$ and $v$ of some face of $\mathcal{F}_q$ such that $x$ and $y$ lie in distinct components of $\mathbb{R}_\infty\setminus\{u,v\}$. Since $x$ and $y$ are both limit points of the sequence $v_1,v_2,\dotsc$, it follows from Lemma~\ref{lemma 6} that the sequence either contains infinitely many terms equal to $u$ or else it contains infinitely many terms equal to $v$.

Suppose now that $x$ and $y$ are adjacent vertices of $\mathcal{F}_q$, which, after applying an element of $G_q$, we can assume are $0$ and $\infty$. Since the sequence of convergents accumulates at $0$ it contains infinitely many terms inside one of the intervals $[-\lambda_q,0]$ or $[0,\lambda_q]$, and it also contains infinitely many terms that lie outside the union of these two intervals. It follows from Lemma~\ref{lemma 6} that the sequence contains infinitely many equal terms, all equal to one of $-\lambda_q$, $0$, or $\lambda_q$.
\end{proof}

Let us now discuss infinite \emph{geodesic} Rosen continued fractions, which were defined in the introduction. An infinite path $\langle v_0,v_1,\dotsc\rangle$  in $\mathcal{F}_q$ is said to be a   \emph{geodesic path} if $\langle v_0,v_1,\dotsc,v_n\rangle$ is a geodesic path for each positive integer $n$. It follows that $[b_1,b_2,\dots]_q$ is a geodesic Rosen continued fraction if and only if its path of convergents is a geodesic path. We can determine whether $[b_1,b_2,\dots]_q$ is a geodesic Rosen continued fraction  using Theorems~\ref{theorem 3} and~\ref{theorem 12}.

The next observation is an immediate corollary of Theorem~\ref{theorem 13}.

\begin{corollary}\label{corollary 4}
Every infinite geodesic Rosen continued fraction converges.
\end{corollary}

Earlier we proved that, given vertices $x$ and $y$ of $\mathcal{F}_q$, we can construct a geodesic path from $x$ to $y$ by iterating the map $\alpha_y$. In fact, we can define $\alpha_y$ even when $y$ is not a vertex of $\mathcal{F}_q$. In this case, for any vertex $x$, the iterates $x, \alpha_y(x), \alpha^2_y(x),\dotsc$ form an infinite path in $\mathcal{F}_q$. We can see that this is a geodesic path because, for any positive integer $n$, we can choose a vertex $y_0$ of $\mathcal{F}_q$ that is sufficiently close to $y$ in the spherical metric on $\mathbb{R}_\infty$ that the two paths $\langle x,\alpha_y(x),\dots,{\alpha_y}^n(x)\rangle$ and $\langle x,\alpha_{y_0}(x),\dots,\alpha_{y_0}^n(x)\rangle$ are identical, and we know that the latter is a geodesic path.  Corollary~\ref{corollary 4} now tells us that the sequence $x, \alpha_y(x), \alpha^2_y(x),\dotsc$ converges in $\mathbb{R}_\infty$, and a short argument that  we omit shows that the limit must be $y$. This gives us the following theorem.


\begin{theorem}\label{theorem 14}
Given a vertex $x$ of $\mathcal{F}_q$ and a real number $y$ that is not a vertex of $\mathcal{F}_q$, the infinite path $\langle x,\alpha_y(x),{\alpha_y}^2(x)\dots,\rangle$ is a geodesic path in $\mathcal{F}_q$ which converges in $\mathbb{R}_\infty$ to $y$.
\end{theorem}

We can now complete the proof of Theorem~\ref{theorem 1}, the first part of which was proved in Section~\ref{section 3}.

\begin{proof}[Proof of Theorem~\ref{theorem 1}: infinite case]
The path $\langle \infty,\alpha_\infty(y),{\alpha_\infty}^2(y)\dots,\rangle$ is the path of convergents of the Rosen continued fraction expansion of $y$ obtained by applying the nearest-integer algorithm. We have seen this already when $y$ is a vertex of $\mathcal{F}_q$, and the same is true when $y$ is not a vertex. Theorem~\ref{theorem 14} says that this path is a geodesic path, so we have proved  Theorem~\ref{theorem 1} for infinite  Rosen continued fractions.
\end{proof}

We finish here with an example to show that a real number may have infinitely many infinite geodesic Rosen continued fraction expansions. In our example $q=4$, but there are similar examples for other values of $q$. The simplest way to describe the example is using the infinite $q$-chain suggested by Figure~\ref{figure 22}. The number $y$ is equal to $[2,2,\dots]_4$, and we can see from the infinite $q$-chain that there are infinitely many geodesic paths from $\infty$ to $y$.

\begin{figure}[ht]
\centering
\includegraphics{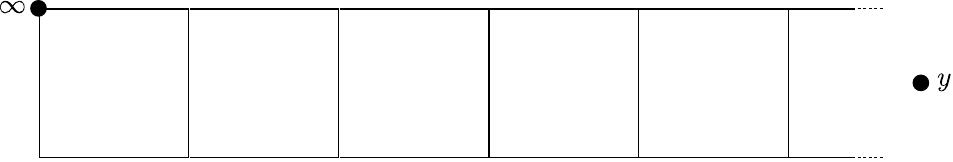}
\caption{There are infinitely many geodesic paths from $\infty$ to $y$.}
\label{figure 22}
\end{figure}

\section{The theta group}\label{section 9}

The Hecke groups $G_q$ are in fact only a countable collection from a larger class of Fuchsian groups that are also known as Hecke groups. To describe this class, let
\[
\sigma(z)=-\frac{1}{z}\quad\text{and}\quad \tau(z)=z+\lambda,
\]
where $\lambda>0$. Hecke proved (see \cite{Ev1973}) that the group generated by $\sigma$ and $\tau$ is discrete if and only if either $\lambda=2\cos (\pi/q)$, $q=3,4,\dotsc$, or $\lambda\geq 2$. The groups with $\lambda<2$ are the Hecke groups $G_q$. In this section we discuss (without details) the group with $\lambda=2$, which we denote by $G_\infty$ because $2\cos(\pi/q)\to 2$ as $q\to\infty$. This group is commonly known as the \emph{theta group}. The Hecke groups with $\lambda>2$ are those Hecke groups  of the second kind (their limit sets are not dense in $\mathbb{R}_\infty$). The Farey graphs for these groups are similar to that of $G_\infty$, and we do not consider these groups any further.

The theta group $G_\infty$ can be described explicitly as the collection of those M\"obius transformations $z\mapsto (az+b)/(cz+d)$, where $a,b,c,d\in\mathbb{Z}$ and $ad-bc=1$, such that
\[
\begin{pmatrix}a & b \\ c& d\end{pmatrix}\equiv \begin{pmatrix}1 & 0 \\ 0& 1\end{pmatrix}\,\text{or}\,\begin{pmatrix}0 & 1 \\ 1& 0\end{pmatrix} \pmod{2}
\]
(see \cite[Corollary~4]{Kn1970}). It is a subgroup of index 3 in the modular group $G_3$. We can define a Farey graph $\mathcal{F}_\infty$ for $G_\infty$ just as we did for the Hecke groups $G_q$ (either by using the fundamental domain of a suitable normal subgroup of $G_\infty$ or using the orbit of the hyperbolic line between $0$ and $\infty$ under $G_\infty$). The Farey graph $\mathcal{F}_\infty$ is illustrated in Figure~\ref{figure 20}.

\begin{figure}[ht]
\centering
\includegraphics{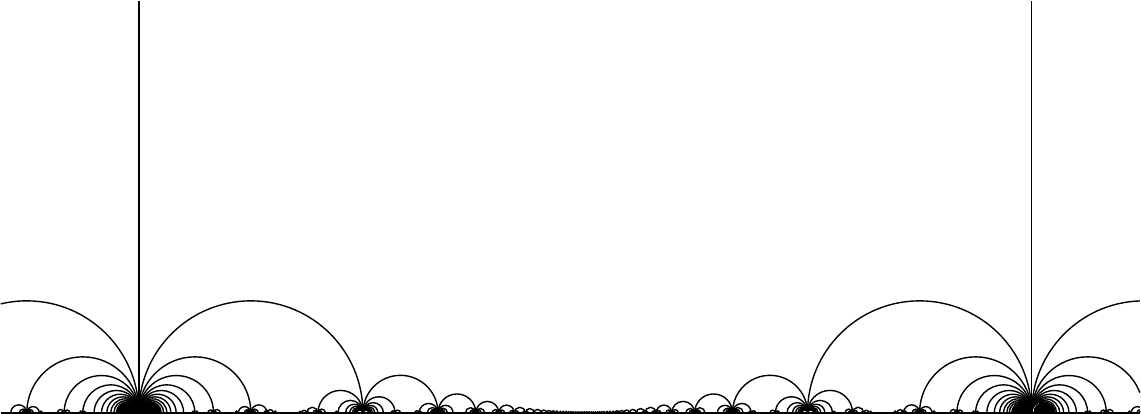}
\caption{The Farey graph $\mathcal{F}_\infty$}
\label{figure 20}
\end{figure}

Unlike the Farey graphs we have met so far, the graph $\mathcal{F}_\infty$ is a tree. Since we know what the elements of $G_\infty$ are explicitly, we can determine the vertices of $\mathcal{F}_\infty$: they are $\infty$ and rational numbers $a/b$, where $a$ and $b$ are coprime integers such that one of them is odd and the other is even. Two vertices $a/b$ and $c/d$ are connected by an edge if and only if $|ad-bc|=1$. From this we see that $\mathcal{F}_\infty$ is a subgraph of the usual Farey tessellation $\mathcal{F}_3$.

Since $\lambda=2$, a Rosen continued fraction of the form
\[
b_1\lambda +\cfrac{-1}{b_2\lambda
          + \cfrac{-1}{b_3\lambda
          +\cfrac{-1}{\raisebox{-1ex}{$\dotsb + b_n\lambda$}}}}\,,
\]
is just a continued fraction with coefficients that are even integers. We denote it by $[b_1,\dots,b_n]_\infty$. Continued fractions with even-integer coefficients have been studied elsewhere; see, for example, \cite{KrLo1996}. These continued fractions correspond to paths in $\mathcal{F}_\infty$ in the same way that Rosen continued fractions $[b_1,\dots,b_n]_q$ correspond to paths in $\mathcal{F}_q$. Since $\mathcal{F}_\infty$ is a tree we immediately obtain the following strong characterisation of geodesic paths.

\begin{theorem}\label{theorem 11}
Each vertex of $\mathcal{F}_\infty$ has a unique geodesic Rosen continued fraction expansion. Furthermore, the continued fraction $[b_1,\dots,b_n]_\infty$ is a geodesic Rosen continued fraction if and only if $b_i\neq0$ for $i=2,\dots,n$.
\end{theorem}

To find the unique geodesic Rosen continued fraction expansion of a vertex $x$ of $\mathcal{F}_\infty$, you can apply the nearest-\emph{even}-integer algorithm to $x$. This algorithm will also give you a geodesic expansion when $x$ is not a vertex of $\mathcal{F}_\infty$; however, in this case there may be other geodesic expansions of $x$. For example,
\[
1=2 +\cfrac{-1}{2
          + \cfrac{-1}{2
          +\dotsb}}
= \cfrac{-1}{-2
          + \cfrac{-1}{-2
          +\dotsb}}\, .
\]

\end{document}